\documentclass{amsart}

\usepackage[utf8]{inputenc}
\usepackage[T1]{fontenc}
\usepackage[english]{babel}
\usepackage{color}
\usepackage{amsmath} 
\usepackage{amssymb} 
\usepackage{amsthm}
\usepackage{graphicx}
\usepackage[colorlinks=true, linkcolor=blue]{hyperref}
\usepackage{cancel}

\usepackage{lmodern}
\usepackage{microtype}

\theoremstyle{plain}
\newtheorem{thm}{Theorem}[section]
\newtheorem*{thm*}{Theorem}

\newtheorem{lem}[thm]{Lemma}

\newtheorem*{cor*}{Corollary}

\newtheorem{rem}{Remark}

\newcommand {\R} {\mathbb{R}} \newcommand {\Z} {\mathbb{Z}}
\newcommand {\T} {\mathbb{T}} \newcommand {\N} {\mathbb{N}}

\newcommand {\p} {\partial}
\newcommand {\dt} {\partial_t}
\newcommand {\hl}{\sigma}

\DeclareMathOperator{\hyp}{_2F^1}

\begin{document}
\title[On the Boussinesq equations]{On the Boussinesq equations with non-monotone
  temperature profiles}

\begin{abstract}
  In this article we consider the asymptotic stability of the two-dimensional Boussinesq
  equations with partial dissipation near a combination of Couette flow and
  temperature profiles $T(y)$.
  As a first main result we show that if $T'$ is of size at most $\nu^{1/3}$ in
  a suitable norm, then the linearized Boussinesq equations with only vertical
  dissipation of the velocity but not of the temperature are stable.
  Thus, mixing enhanced dissipation can suppress Rayleigh-B\'enard instability in
  this linearized case.

  We further show that these results extend to the (forced) nonlinear equations
  with vertical dissipation in both temperature and velocity.
\end{abstract}

\author{Christian Zillinger}
\address{Karlsruhe Institute of Technology, Department of Mathematics, Englerstrasse.
  2, 76131 Karlsruhe, Germany}
\email{zillinger@kit.edu}
\maketitle

\tableofcontents

\section{Introduction}
\label{sec:introduction}  
The Boussinesq equations are a standard approximate model of heat transfer in (viscous)
fluids and are given by a coupled system of the Navier-Stokes equations and
a dissipative transport equation for the temperature density:
\begin{align}
  \label{eq:B}
  \begin{split}
  \dt v+ v \cdot \nabla v + \nabla p &= (\nu_x \p_x^2 + \nu_y \p_y^2 )v + \theta e_2, \\
  \dt \theta + v \cdot \nabla \theta &= (\mu_x \p_x^2 + \mu_y \p_y^2) \theta,\\
  \nabla \cdot v &=0.
  \end{split}
\end{align}
Here $v \in \R^2$ denotes the velocity, $p \in \R$ is the pressure, $\theta \in
\R$ is the temperature and we consider the domain $\T \times \R \ni (x,y)$.
The $\theta e_2$ term models buoyancy which causes hotter fluid to rise and colder fluid
to sink.

In Sections \ref{sec:model} and \ref{sec:hydroimbalance} of this article we
consider the setting with only vertical dissipation of the velocity,
\begin{align*}
 \nu_x=\mu_x=\mu_y=0, \nu_y=:\nu>0. 
\end{align*}
We refer to this setting as vertical dissipation.
In Section \ref{sec:nonlinear} we assume vertical dissipation in both velocity and
temperature, which we refer to as full vertical dissipation.

One readily observes that at least formally any pair of functions of the form
\begin{align}
  v= (\beta y, 0), \theta=T(y),
\end{align}
with $\beta \in \R$ and $T$ smooth are automatically stationary solutions of the
vertical dissipation problem (choosing $p=p(y)$ suitably).
Here a particular focus in existing results has been on the case when $T$ is
affine and increasing, that is hotter fluid is on top of colder fluid, which is
known as hydrostatic balance.
A main aim of this article is to study more general profiles $T(y)$ and in
particular answer how much $T$ may oscillate if both shear and viscosity are
available to counteract thermal instability.

More generally, the problem of partial dissipation has been an area of extensive research, where we
in particular mention the recent works
\cite{elgindi2015sharp,widmayer2018convergence,doering2018long,yang2018linear,wu2019stability,deng2020stability,wu2020stabilizing} and \cite{dong2020stability2}. The question of global
wellposedness has been addressed in series of works by Chae, Nam and Kim \cite{chae1999local,chae2006global}.

In this article, we we will focus on questions of asymptotic stability close to
specific families of solutions and how the interaction of mixing and temperature
stratification may counteract instability.

In \cite{yang2018linear} Yang and Lin studied the stability of the linearized \emph{inviscid}
problem around the case where $T(y)=\alpha y$ is affine and showed that for
some stability results it is necessary that $\alpha>0$ and thus $T$ is
increasing.
We recall these results in Section \ref{sec:model} and emphasize that the
threshold with respect to $\alpha$ depends on whether one studies
\begin{itemize}
\item the vorticity $\omega$, which is always unstable, 
\item the horizontal component of the velocity $v_1$, which is stable if
  $\alpha>0$ and unstable if $\alpha<0$, or 
\item the vertical component of the velocity $v_2$, which is stable if $\alpha>
  -2$ and unstable if $\alpha<-2$.
\end{itemize}
Thus, already in this case in a specific sense one may allow $\alpha$ to be
\emph{negative} if it is sufficiently small.

Recently, Masmoudi, Said-Houari and Zhao \cite{masmoudi2020stability} showed that the associated \emph{nonlinear} problem near
$T(y)=\alpha y$, $\alpha>0$ without
thermal diffusion but with viscous diffusion is asymptotically stable in Gevrey
regularity. These results in particular show that this partial dissipation
problem behaves similarly to the Euler equations \cite{bedrossian2015inviscid} instead
of the Navier-Stokes equations \cite{bedrossian2016sobolev}.
If one instead considers full dissipation, in \cite{zillinger2020enhanced} we
adapted the methods of \cite{bedrossian2016sobolev,liss2020sobolev}
to establish nonlinear stability in Sobolev regularity.

This article extends the results of \cite{zillinger2020enhanced} to the case of \emph{negative} $\alpha$
and partial dissipation. More precisely, we show for the linearized problem with vertical dissipation that
the evolution is asymptotically stable provided 
\begin{align}
  \alpha> -\frac{1}{100}\sqrt[3]{\nu}.
\end{align}
Similar results hold for $T(y)$ non-affine.
Thus, mixing enhanced dissipation can suppress Rayleigh-B\'enard instability with
an enhanced dependence on $\nu$. We remark that beneficial interaction of shear
and (in)stability in the context of reaction-diffusion and turbulence has
previously been observed in \cite{spiegel1984reaction,doering1993stability,castaing1989scaling}.

For the nonlinear problem we further show that for affine $T$ and full vertical
dissipation the same stability results hold.
As shown in \cite{masmoudi2020stability} in the case of vertical dissipation
only in the vorticity a more careful
analysis is required to control resonances, reminiscent of echoes in the Euler
equations \cite{dengmasmoudi2018,dengZ2019}.

Our main results concerning the linearized problem are summarized in the following theorem.
\begin{thm}
  \label{thm:main} Let $T: \R \rightarrow \R$ be a given temperature profile.
  Let $N \in \N$ and suppose that $T'(y) \in L^\infty$.
  We then consider linearized Boussinesq equations with vertical dissipation in
  the velocity only around $v=(y,0)$, $\theta=T(y)$ in
  coordinates $(x+ty, y)$:
  \begin{align*}
    \dt \omega &= \nu (\p_y-t\p_x)^2 \omega + \p_x \theta , \\
    \dt \p_x \theta &= T'(y)\p_x v_2, \\
    (t,x,y) &\in (0,\infty)\times \T \times \R 
  \end{align*}
  Then if the Fourier transform of $T'$ satisfies the estimate
  \begin{align}
    \label{eq:27}
    \sup_{\xi} \int |\mathcal{F}(T')(z-\xi)|  (\frac{1+|z|}{1+|\xi|} + \frac{1+|\xi|}{1+|z|})^{N} (1+ \min(\nu^{-\frac{2}{3}}, |z-\xi|^{\frac{2}{3}}))dz < \frac{1}{100} \nu^{1/3},
  \end{align}
  the initial value problem is stable in $H^N \times H^N$ in the sense that
  there exists a constant $C>0$ such that for any initial data $(\omega_{in}, \p_x
  \theta_{in}) \in H^N \times H^N$ the solution satisfies 
  \begin{align*}
    \|\omega(t)\|_{H^N} + \nu \|\p_x \theta\|_{H^N} \leq C \nu^{-\frac{2}{3}} (\|\omega_{in}\|_{H^N} + \nu \|\p_x \theta_{in}\|_{H^N}).
  \end{align*}
  In particular, if $T'(y)=\alpha y$, stability holds if $\alpha> -\frac{1}{100}
  \nu^{1/3}$.
  \end{thm}
The condition \eqref{eq:27} is a sufficient condition to control commutators
involving $T'(y)$ and is probably not optimal in its dependence on $N$. In
the case where $T(y)=\alpha y$ is affine, it reduces to the condition $|\alpha|<
C_N \nu^{1/3}$ and thus allows for $\alpha$ to be negative. See Theorem
\ref{thm:bad} for further discussion. 
 
For the nonlinear problem with full vertical dissipation we obtain similar results.  
\begin{thm}
Let $T: \R \rightarrow \R$ be a given temperature profile and consider the (forced) nonlinear problem around $v=(y,0)$,
$\theta=T(y)$ \emph{with vertical dissipation} $\nu_y=\mu_y=:\nu>0$ in coordinates $(x+ty,y)$:
\begin{align*}
  \dt \omega + v \cdot \nabla_t \omega &= \nu (\p_y-t\p_x)^2 \omega + \p_x \theta , \\
  \dt \theta + v \cdot \nabla_t \theta &= \nu (\p_y-t\p_x)^2 \theta + T'(y) v_2 \omega,\\
  (t,x,y) &\in (0,\infty)\times \T \times \R,
\end{align*}
and suppose that $T'$ satisfies the assumptions of Theorem \ref{thm:main}. Then
this problem is stable in Sobolev regularity.
More precisely, for any $N \in \N$, $N\geq 5$ there exists $\epsilon_N=\epsilon_N(\nu)$
such that if initially
\begin{align*}
  \|\omega\|_{H^N}^2 + \nu^{-1} \|\p_x \theta\|_{H^N}^2 < \epsilon^2 < \epsilon_{N}^2,
\end{align*}
then the solution remains bounded by $10 \nu^{-2/3} \epsilon^2$ for all times.
\end{thm}
We also obtain time integrability results for $v$, $(\p_y-t\p_x) \omega$ and
$(\p_y-t\p_x)\theta$, which are stated in Sections \ref{sec:hydroimbalance} and \ref{sec:nonlinear} and
omitted here for brevity.

\begin{itemize}
\item In the special case when $T(y)=\alpha y$ is affine the assumption reduces
  to $|\alpha|\leq \frac{1}{100}\nu^{1/3}$.  
\item We stress that $\alpha$ here is allowed to be \emph{negative}.
  As a related result in Lemma \ref{lem:basic2} we remark that the \emph{inviscid}
  results of \cite{yang2018linear} extend to $0\geq \alpha>-2$ when considering
  the vertical component of the velocity $v_2$.
\item If there is no shear, then partial dissipation is not sufficient to restore
  stability of the vorticity for $\alpha<0$ (see Lemma \ref{lem:basic}).
\item A combination of shear and vertical dissipation suffices to restore
  stability of the vorticity. Moreover, in that case we obtain an enhanced
  threshold in terms of $-\nu^{1/3}$.
\item These results further extend to the case of a non-affine, oscillating temperature profile
  $T(y)$. In particular, we do not rely on cancellations or conserved quantities
  available in the hydrostatic balance case.
\item In addition to the linearized Boussinesq equations, we obtain results for
  the nonlinear small data problem, however, only with \emph{full vertical}
  dissipation (considering $T(y)$ non-affine as a solution of the forced problem).
  As recently shown in \cite{masmoudi2020stability} this stronger assumption is probably
  necessary for stability in Sobolev regularity, since otherwise resonance
  chains may yield norm inflation.
% \item As a complementary result to the linear stability, we show that for an
%   oscillating profile of the form $T(y)=c_1 \cos( c_2 y)$ with
%   suitable choice of $c_1,c_2$ resonances occur which break stability. Thus, we
%   show that in this sense smallness conditions depending on $\nu$ are necessary
%   to rule out exponential instabilities.
\end{itemize}

The remainder of the article is structured as follows:
\begin{itemize}
\item In Section~\ref{sec:model} we recall some results for the \emph{inviscid}
  problem, first obtained in \cite{yang2018linear}, to introduce instability mechanisms and to discuss in which sense
  (partial) dissipation is necessary for stability results. With these
  motivations we formulate four main questions Q1-Q4, which we address
  throughout the article.
\item In Section~\ref{thm:good} we begin by studying the special case when
  $T(y)$ is affine, where arguments are more transparent.
  In particular, we show that here the slope of the temperature profile can be allowed to be
  \emph{negative} (colder fluid on top of hotter fluid) and that the size of the threshold
  depends on $\nu$ with an enhanced rate. 
\item In Section~\ref{thm:bad} we extend these linear results to
  the case of a general temperature profile $T(y)$ satisfying suitable smallness
  conditions. In particular, $T$ is allowed to oscillate.
\item Building on the linearized results, in Section \ref{sec:nonlinear} we
  study the nonlinear small data problem.
  Due to possible resonance chains, we here instead consider full vertical
  dissipation and consider $T(y)$ as a solution of the forced problem.
  This extends previous nonlinear results in \cite{zillinger2020enhanced} for
  the affine, increasing case to possibly oscillating profiles.
\end{itemize}

\subsection{Notation}
\label{sec:notation}

Throughout this article we consider solutions of the Boussinesq equations with
vertical dissipation near the stationary solution
\begin{align*}
  v&=
  \begin{pmatrix}
    y \\0
  \end{pmatrix}, \theta = T(y),\\
(t,x,y)&\in (0,\infty) \times \T \times \R.
\end{align*}
In this setting it is natural to work in coordinates moving with the flow
\begin{align*}
  (x+ty,y)
\end{align*}
and consider the equations satisfied by the perturbations in these coordinates.

If there is no possibility of confusion these perturbations are again denoted as
$\omega$ and $\theta$ and the linearized problem studied in Section \ref{sec:hydroimbalance} is given by 
\begin{align*}
  \dt \omega &= \nu (\p_y-t\p_x)^2 \omega + \p_x \theta, \\
  \dt \theta &= T'(y)v_2,\\
  v&= \nabla_t^{\perp}\Delta_t^{-1}\omega,\\
(t,x,y)&\in (0,\infty) \times \T \times \R.
\end{align*}
where
\begin{align*}
  \nabla_t =
  \begin{pmatrix}
    \p_x \\ \p_y-t \p_x
  \end{pmatrix},
  \Delta_t= \p_x^2 + (\p_y-t\p_x)^2
\end{align*}
are the gradient and Laplacian in these coordinates.

In the nonlinear problem considered in Section \ref{sec:nonlinear} we
additionally assume vertical dissipation also in the temperature and interpret
$T(y)$ as a solution of the forced problem.
The system satisfied by the perturbation in coordinates moving with the shear is
then given by
\begin{align*}
  \dt \omega &= \nu (\p_y-t\p_x)^2 \omega +\p_x \theta - v \cdot \nabla_t \omega, \\
  \dt \theta &= T'(y)v_2 - v \cdot \nabla_t \theta,\\
  v&= \nabla_t^{\perp}\Delta_t^{-1}\omega,\\
(t,x,y)&\in (0,\infty) \times \T \times \R.
\end{align*}

We denote the Fourier transform of a function $u(x,y) \in L^2 (\T \times \R)$ by
$\tilde{u}(k,\xi) \in L^2 (\Z \times \R)$ or $\mathcal{F} u$. Furthermore, we
study several Fourier multipliers (see \eqref{eq:22}), including
\begin{align*}
  A(T,\xi, k) &= \exp(-2\int_0^T\frac{1}{1+(\frac{\xi}{k}-t)^2} dt), \\
    B(T,k,\xi) &= \exp(-2 \int_0^T\frac{1}{\sqrt{1+(\frac{\xi}{k}-t)^2}}1_{I}(t) dt),
\end{align*}
with $I$ being a prescribed time interval/Fourier region (see \eqref{eq:26}):
\begin{align*}
  I= \{t\geq 0 : |\frac{\xi}{k}-t|\leq C\},
\end{align*}
and $C$ proportional to $\nu^{-1/3}$.

In Section \ref{sec:nonlinear} we study energy estimates on a given time
interval $(0,T)$. Since $T$ is fixed throughout this section, we omit it from
our notation and for instance write
\begin{align*}
 \|u\|_{L^pH^N}:= \| \| u(t, \cdot)\|_{H^N}\|_{L^p((0,T))}. 
\end{align*}

We write $a \lesssim b$ if there exists a universal constant $C>0$ such that
$|a|\leq C |b|$.

\section{Model Cases of Instability}
\label{sec:model}
In order to introduce ideas and mechanisms, in this section we recall results
available in the literature for
\begin{itemize}
\item The linearized viscous problem without shear around $\theta=\alpha y$, $v=0$ and
\item The linearized inviscid problem with shear around $\theta=\alpha y$, $v=(y,0)$
\end{itemize}
Here, for simplicity we consider viscous dissipation in both horizontal and
vertical direction, but no thermal dissipation.
As a reference for the isolated mechanisms the interested reader is referred to
the textbook by Frisch and Yaglom \cite[Section 2.8.3]{yaglom2012hydrodynamic}.
We emphasize that the results of this section are not new, but serve to motivate
our questions Q1--Q4 stated at the end of this section, which we address in this
article.
Furthermore, they show that under weaker assumptions instabilities may form and
that the conditions in Theorem \ref{thm:main} are in this sense optimal.

In the case without shear, explicit solutions are available and it is known that
the slope of $\theta$ yields a sharp dichotomy between stability and exponential
instability.
The following basic lemma is reproduced from \cite[Proposition 2.6]{zillinger2020enhanced}.
\begin{lem}
  \label{lem:basic}
Consider the Boussinesq equations in vorticity formulation linearized around
 \begin{align*}
  v=(0,0), \theta= \alpha y,
 \end{align*}
 where $\alpha \in \R$:
\begin{align}
  \label{eq:constcoeffcase}
  \begin{split}
  \dt \omega &= \nu \Delta \omega + \p_x \theta, \\
  \dt \theta + \alpha v_2 &= \mu \Delta \theta.
  \end{split}
\end{align}
Here $v_2$ denotes the vertical component of the velocity field.
Further suppose that at least one of $\nu$ or $\mu$ is zero.
The the evolution is stable if $\alpha>0$ in the sense that for every $N \in \N$ the energy 
\begin{align}
  \alpha \|\omega\|_{H^N}^2 + \|\nabla \theta\|_{H^N}^2
\end{align} 
is decreasing.
In contrast, if $\alpha<0$, there exist solutions which grow exponentially in time.
\end{lem}
As we show in Lemma \ref{lem:basic2} when adding shear the instability for
$\alpha<0$ is significantly reduced and the evolution of $v_2$ is even
asymptotically stable if $\alpha$ is not too large.
\begin{proof}
  In the interest of accessibility, we reproduce the main steps of the proof from \cite{zillinger2020enhanced}.

We observe that equation \eqref{eq:constcoeffcase} is a constant coefficient PDE and hence we obtain a decoupled system of ODEs for each Fourier mode with respect to $x$ and $y$:
\begin{align}
  \label{eq:3}
  \dt
  \begin{pmatrix}
    \tilde{\omega} \\
    \tilde{\theta}
  \end{pmatrix}
=
  \begin{pmatrix}
    -\nu (k^2+\xi^2) & ik \\
    \frac{ik \alpha}{k^2+\xi^2} & -\eta (k^2+\xi^2)
  \end{pmatrix}
   \begin{pmatrix}
    \tilde{\omega} \\
    \tilde{\theta}
   \end{pmatrix},     
\end{align}
where we use $\tilde{\omega}$ to denote the Fourier transform of the vorticity
and $k \in \Z, \xi \in \R$ to denote the Fourier variables.
In particular, we may study the problem at each frequency.
\\

\underline{The case $\alpha>0$:} Let  $(k,\xi)$ and $\alpha>0$ be given. Then we may reformulate the problem as 
\begin{align}
  \dt
  \begin{pmatrix}
    \sqrt{\alpha} \tilde{\omega} \\
    \sqrt{k^2+\xi^2} \tilde{\theta}
  \end{pmatrix}
=
  \begin{pmatrix}
    -\nu (k^2+\xi^2) & \frac{ik\sqrt{\alpha}}{\sqrt{k^2+\xi^2}} \\
    \frac{ik \sqrt{\alpha}}{\sqrt{k^2+\xi^2}} & -\eta (k^2+\xi^2)
  \end{pmatrix}
   \begin{pmatrix}
    \sqrt{\alpha} \tilde{\omega} \\
    \sqrt{k^2+\xi^2} \tilde{\theta}
   \end{pmatrix}.     
\end{align}
Note that the off-diagonal entries are equal and purely imaginary. Therefore, if we denote the matrix by $M$ it holds that $M+\overline{M}^T$ is a real-valued, negative semi-definite diagonal matrix.
Hence it follows that 
\begin{align*}
  \frac{d}{dt} \left|\begin{pmatrix}
    \sqrt{\alpha} \tilde{\omega} \\
    \sqrt{k^2+\xi^2} \tilde{\theta}
  \end{pmatrix} \right|^2 = \overline{\begin{pmatrix}
    \sqrt{\alpha} \tilde{\omega} \\
    \sqrt{k^2+\xi^2} \tilde{\theta}
  \end{pmatrix}} \cdot (M+\overline{M}^T) \begin{pmatrix}
    \sqrt{\alpha} \tilde{\omega} \\
    \sqrt{k^2+\xi^2} \tilde{\theta}
  \end{pmatrix} \leq 0.
\end{align*}
Integrating this estimate with respect to $\xi$ and $k$ (possibly with respect to a weight $\langle(k,\xi)\rangle^N$) it follows that 
\begin{align*}
  \alpha \|\tilde{\omega}\|_{L^2}^2 + \|\sqrt{k^2+\xi^2}\tilde{\theta}\|_{L^2}^2
\end{align*}
is non-increasing. The claimed result thus follows by Plancherel's theorem.
\\

\underline{The case $\alpha<0$:} Let $(k,\xi)$ with $k\neq 0$ and $\alpha<0$ be given.
Then the eigenvalues of the matrix
  \begin{align*}
   \begin{pmatrix}
    -\nu (k^2+\xi^2) & ik \\
    \frac{ik \alpha}{k^2+\xi^2} & -\eta (k^2+\xi^2)
  \end{pmatrix}   
  \end{align*}
  are given by
  \begin{align*}
    \lambda_{1,2}&= -\frac{\nu+\mu}{2} (k^2+\xi^2) \pm \sqrt{(\frac{\nu+\mu}{2} (k^2+\xi^2))^2 - \nu \eta (k^2+\xi^2)^2 - \alpha \frac{k^2}{k^2+\xi^2}}\\
    &=-\frac{\nu+\mu}{2}(k^2+\xi^2) \pm \sqrt{\left( \frac{\nu-\mu}{2}(k^2+\xi^2) \right)^2 - \alpha \frac{k^2}{k^2+\xi^2}},
  \end{align*}
  where we used the binomial formula $(a+b)^2-4ab=(a-b)^2$ in the last step.
  
  We recall that by assumption (at least) one of $\nu$, $\eta$ vanishes.
  Therefore we define $C=\max(\nu, \eta)$ and observe that
  \begin{align*}
    (\eta+\nu)^2 = (\eta-\nu)^2 =:C^2
  \end{align*}
  and that
  \begin{align*}
    \lambda_1 = - C (k^2+\xi^2) + \sqrt{C^2 (k^2+\xi^2) + (-\alpha)\frac{k^2}{k^2+\xi^2}},
  \end{align*}
  is strictly positive, since $(-\alpha)\frac{k^2}{k^2+\xi^2}$ is positive. 
  This matrix thus has a positive eigenvalue and there exist solutions of
  \eqref{eq:3} which grow exponentially in time.
  Given these exponentially growing solutions on single Fourier modes, we next
  construct exponentially growing solutions in $H^N$.
  We may pick a compact set in Fourier space, e.g. a ball, and construct
  initial data $(\omega_0, \theta_0) \in H^N\times H^{N+1}$ by prescribing the
  Fourier transform of the initial data to match these solutions (and vanish outside
  the ball). The corresponding solution then also exhibits exponential growth in
  time.
\end{proof}
We remark that in the inviscid case, $\nu=\mu=0$, these eigenvalues further simplify to 
\begin{align*}
  \pm \sqrt{-\alpha \frac{k^2}{k^{2}+\xi^2}},
\end{align*}
which are either purely imaginary if $\alpha>0$ or
positive and negative if $\alpha<0$.
Thus, if $\alpha>0$ (hotter fluid is above) the evolution is not
exponentially unstable. One speaks of \emph{hydrostatic equilibrium}. The
stability of this solution in the inviscid setting has recently been studied in
\cite{elgindi2015sharp,widmayer2018convergence}.

In contrast if $\alpha<0$ (that is, the fluid is hotter below) then one
eigenvalue is positive and the solution is exponentially unstable.
This phenomenon is known as \emph{Rayleigh-B\'enard} instability. One main question in the following will then be whether
a shear flow can suppress this instability.

Having discussed the effects of dissipation without shear. We next consider the
effects of an affine shear in the inviscid problem, where again explicit
solutions are available.
The following results have been previously obtained in \cite{yang2018linear,
  zillinger2020enhanced,masmoudi2020stability} for $\alpha>0$. By minor
modifications of the proof the results further extend to negative $\alpha$ and
higher Sobolev norms.
\begin{lem}
  \label{lem:basic2}
  Consider the linearized inviscid Boussinesq equations in vorticity formulation around
  \begin{align*}
    v=(y,0), \theta=\alpha y
  \end{align*}
  and coordinates $(x+ty,y)$ moving with the the shear.
  Furthermore, define $c= \frac{1}{2} \Re(\sqrt{1-4\alpha}) \in [0,\infty)$.
  Then the velocity and temperature satisfy the following estimates
  \begin{align*}
    \|\theta\|_{H^N}&\lesssim t^{-1/2+c}(\|\omega_0\|_{H^{N+1}}+ \|\theta_0\|_{H^{N+2}}),\\
    \|v_1-\langle v_1 \rangle\|_{H^N}&\lesssim t^{-1/2+c}( \|\omega_0\|_{H^{N+1}}+ \|\theta_0\|_{H^{N+2}}),\\
    \|v_2\|_{H^N}&\lesssim t^{-3/2+c}( \|\omega_0\|_{H^{N+2}}+ \|\theta_0\|_{H^{N+3}}),
  \end{align*}
  and are thus stable if $c<\frac{1}{2}$ ($\alpha>0$) and  $c<\frac{3}{2}$
  ($\alpha>-2$), respectively.
  They are unstable if $c>\frac{1}{2}$ ($\alpha<0$) or $c> \frac{3}{2}$ ($\alpha<-2$).

  The evolution of the vorticity in contrast is unstable for all $\alpha$ in the
  sense that there exists non-trivial initial data such that
  \begin{align*}
    \|\omega(t)\|_{H^N}&\geq C t^{1/2+c}( \|\omega_0\|_{H^N}+ \|\p_x \theta_0\|_{H^{N}})
  \end{align*}
  as $t\rightarrow \infty$.
\end{lem}
We emphasize that for $v_2$ we may
allow $0>\alpha>-2$ to be negative and that the evolution of the vorticity $\omega$ is unstable for
any $\alpha$.

We remark that this combination of stability and instability is consistent with
the Orr mechanism.
More precisely, by an integration by parts argument it holds that
\begin{align*}
  \|v_1-\int v_1 dx \|_{L^2}\leq C t^{-1} \|\omega(t)\|_{H^{1}}.
\end{align*}
Hence, if the velocity is asymptotically stable with a sharp decay rates of
for instance $t^{-1/2}$, this implies that the vorticity is algebraically unstable in $H^{1}$
with a growth rate at least $t^{-1/2+1}=t^{1/2}$.

\begin{proof}
  As in the proof of Lemma \ref{lem:basic} we consider the Fourier formulation,
  now in coordinates $(k, \xi+kt)$ moving with the shear:
  \begin{align*}
    \dt
    \begin{pmatrix}
      \tilde{\omega}\\ \tilde{\theta}
    \end{pmatrix}
=
    \begin{pmatrix}
      0 & ik \\
      \frac{ik \alpha}{k^2+(\xi-kt)^2} & 0 
    \end{pmatrix}
                                         \begin{pmatrix}
                                           \tilde{\omega}\\ \tilde{\theta}
                                         \end{pmatrix}.
  \end{align*}
  Due to the vanishing diagonal structure, we may decouple this problem as
  \begin{align*}
    \dt^2 \tilde{\omega}&= - \frac{\alpha k^2}{k^2+(\xi-kt)^2} \tilde{\omega}, \\
    \dt \tilde{\omega} &= ik \tilde{\theta}.
  \end{align*}
  After relabeling and shifting time by $\frac{\xi}{k}$, we observe that the first equation
  corresponds to a Schrödinger equation with potential:
  \begin{align*}
    (\dt^2 + \frac{\alpha}{1+t^2}) u =0.
  \end{align*}
  As observed in \cite{yang2018linear} this problem can be solved explicitly in terms of hypergeometric functions:
  \begin{align}
    \label{eq:1}
    \begin{split}
    u(t)&= c_1 \hyp(-\frac{1}{4}-\frac{1}{4}\sqrt{1-4\alpha}, -\frac{1}{4}+\frac{1}{4}\sqrt{1-4\alpha}, \frac{1}{2}, -t^2) \\
    &\quad + c_2 \ t\  \hyp(\frac{1}{4}-\frac{1}{4}\sqrt{1-4\alpha}, \frac{1}{4}+\frac{1}{4}\sqrt{1-4\alpha}, \frac{3}{2}, -t^2).
    \end{split}
  \end{align}
  As $t\rightarrow \infty$, it holds that $\hyp(a,b,c,-t^2)\sim C t^{-2a}$ (see
  \cite[15.8(ii)]{NIST:DLMF}).
  The same asymptotic behavior is exhibited by the approximate problem
  \begin{align*}
    (\dt^2 + \frac{\alpha}{t^2}) f=0,
  \end{align*}
  which we use to simplify discussion in the following.
  Making the ansatz $f=t^{\beta}$, we obtain that
  \begin{align}
    \label{eq:4}
    \begin{split}
    f&= c_1 t^{\beta_1} + c_2 t^{\beta_2}, \\
    \beta_{1,2}&= \frac{1}{2}(1\pm \sqrt{1-4\alpha}),
    \end{split}
  \end{align}
  which matches the asymptotic behavior of the hypergeometric functions in
  \eqref{eq:1}.
  In particular, we observe that for any $\alpha$, $\beta_1$ has positive real
  part which results in an algebraic instability of $f$ and hence $\tilde{\omega}$.
  When considering the velocity and temperature, we recall that
  \begin{align*}
    ik \theta = \p_t \omega \sim \p_t f
  \end{align*}
  and that the Biot-Savart law combined with the shear by $(y,0)$ provides a
  gain of $t^{-1}$ for $v_1-\langle v_1 \rangle$ and by $t^{-2}$ for $v_2$ by
  the Orr mechanism.
  Hence, we deduce that
  \begin{align*}
    \|v_1-\langle v_1 \rangle\|_{H^N} + \|\theta\|_{H^N} \sim t^{\beta_1 -1}, \\
    \|v_2\|_{H^N}\sim t^{\beta_1-2}
  \end{align*}
  with $\beta_1$ as in \eqref{eq:4}.
  In particular, we observe that
  \begin{align*}
    \Re(\beta_1 -1)= c + \frac{1}{2} <0 \text{ if } \alpha >0, \\
    \Re(\beta_1 -1)= c + \frac{1}{2} >0 \text{ if } \alpha<0, \\
    \Re(\beta_1 -2)= c + \frac{3}{2} <0 \text{ if } \alpha >-2, \\
    \Re(\beta_1 -2)= c + \frac{3}{2} >0 \text{ if } \alpha <-2,
  \end{align*}
  where we used that $\sqrt{1-4\alpha}=1 \Leftrightarrow \alpha=0$ and
  $\sqrt{1-4\alpha}=3 \Leftrightarrow \alpha=-2$.
\end{proof}

Given these (in)stability results our main questions in this article are the following:
\begin{enumerate}
\item[Q1] \label{Q1} How much dissipation (and in which directions) needs to be added to
  restore linear stability?
\item[Q2] \label{Q2} Can we allow $\alpha$ to be negative and how does the threshold depend on the
  dissipation?
\item[Q3] \label{Q3} When considering the problem without thermal dissipation, it is natural to
  consider the more general problem around $v=(y,0)$, $\theta=T(y)$. Under which
  conditions on $T$ are such solutions linearly stable? For instance, can we
  allow $T$ to oscillate?
\item[Q4] \label{Q4} Do these results extend to the nonlinear small data regime and if so how
  do stability regions depend the dissipation coefficients (that is, what perturbations can be considered ``small'')?
\end{enumerate}
In this paper we focus on the case without thermal dissipation and $v=(y,0)$,
$\theta=T(y)$.
The converse problem without viscous dissipation and $v=(U(y),0)$,
$\theta=\alpha y$ or time-dependent shear and temperature profile could be of future interest.
We address questions Q1 and Q2 in Section \ref{sec:linear} and Q3 in Section
\ref{sec:nonaffine}.
The question Q4 of nonlinear stability is addressed in Section \ref{sec:nonlinear}.

\section{Shear can Counteract Hydrostatic Imbalance}
\label{sec:hydroimbalance}
Building on the results of Lemma \ref{lem:basic2} for a combination of Couette
flow and an unstable affine temperature profile, in this section we consider the
problem with partial dissipation.

More precisely, we consider the nonlinear Boussinesq equations with vertical
dissipation of the velocit and without thermal diffusion:
\begin{align*}
  \dt v + v \cdot \nabla v + \nabla p &= \nu \p_y^2 v +
  \begin{pmatrix}
    0 \\ \theta 
  \end{pmatrix}, \\
  \dt \theta + v \cdot \nabla \theta &= 0.
\end{align*}
As remarked in the introduction, in Section \ref{sec:nonlinear} we additionally
impose vertical thermal diffusion, but do not require it for the linear
stability results of this section.

We observe that for any $\beta \in \R$ and any function $T(y)$,
the collection
\begin{align*}
  v&=
  \begin{pmatrix}
     \beta y \\ 0
  \end{pmatrix},\\
   \theta&= T(y),\\
    p&=\int^y T(s)ds,
\end{align*}
is a stationary solution of these equations.
As remarked in Section \ref{sec:model} it is natural to ask about the stability of such solutions.

In Section \ref{sec:model} we studied some related special cases when $T(y)$ is
affine:
\begin{itemize}
\item In Lemma \ref{lem:basic} we studied the problem with trivial shear, that is $\beta=0$.
In this setting the flow turned out to be linearly stable if $T$ is increasing
and linearly exponentially unstable if $T$ is decreasing, even if the slope is
very small.
\item In Lemma \ref{lem:basic2} we instead considered the case with shear but with
trivial dissipation and saw that while the exponential instability is reduced to
an algebraic one, the evolution of the vorticity is unstable.
\end{itemize}
The aim of this article is to understand how these results change when adding
partial dissipation and whether they extend to more general profiles $T$.
In this section we study the linearized problem first for the case of $T$ affine
(answering questions Q1, Q2)
and then for general $T$ in Section \ref{sec:nonaffine} (answering Q3).
The nonlinear problem with full vertical dissipation is discussed in Section
\ref{sec:nonlinear}, which answers Q4.
The author would like to thank Charlie Doering for raising the question of the
stability of pairs $v=(U(y),0), \theta = T(y)$ in a discussion.

\subsection{Affine Temperature}
\label{sec:linear}
In order to introduce ideas and mechanisms, we first study the case
\begin{align*}
  T(y)=\alpha y,
\end{align*}
where we allow $\alpha\in \R$ to be \emph{negative} with a threshold depending on
$\nu$.
More precisely, it turns out that for this special linearized problem we may
allow $\alpha$ to be arbitrarily large, but for the nonlinear setting of Section
\ref{sec:nonlinear} and the non-affine problem we require a bound by
$\nu^{1/3}$. Shear enhanced dissipation suppresses Rayleigh-B\'enard instability
in this case, thus answering questions Q1 and Q2 of Section \ref{sec:model}.

We remark that results for $\alpha$ positive have been previously established in \cite{zillinger2020enhanced}.
As the main novelties of this article, we show that even if $\alpha$ is negative
(but small) stability holds and that we may further allow $T$ to be non-affine
(see Section \ref{sec:nonaffine}).

\begin{thm}
    \label{thm:good}
    Consider the linearized Boussinesq equations around $v=(y,0), \theta=\alpha y$ in coordinates 
    \begin{align*}
      (t, x-ty, y)
    \end{align*}
    moving with the shear flow:
    \begin{align*}
      \dt \omega &= \nu (\p_y-t\p_x) \omega + \p_x\theta, \\
    \dt \theta &= \alpha v_2,
    \end{align*}
    on the domain $\T \times \R$.
    Then there exists $\alpha_*=-\frac{1}{100} \sqrt[3]{\nu} <0$ such that the
    linearized evolution is stable at the level of the vorticity for any
    $\alpha$ with $|\alpha|< \alpha_*$.
    More precisely, for any $N\in \N$ there exists a constant $0<C=C(\nu,\alpha)$ such that for all times $t>0$ it holds that
    \begin{align*}
      \|\omega(t)\|_{H^N}+ \|\p_x\theta(t)\|_{H^{N}}\leq C  \|\omega_0\|_{H^N}+  C \|\p_x \theta_0\|_{H^{N}},
    \end{align*}
    where $\omega_0, \theta_0$ denote the initial data.
\end{thm}

We stress that here we can allow $\alpha$ to be negative and that for $0< \nu
<1$, the threshold $\nu^{1/3}$ is improved compared to the dissipative scale.

In particular, for this special setting we may even consider $\alpha\in \R$
arbitrary, but in view of later results focus on the case of small negative
$\alpha$.

  \begin{proof}[Proof of Theorem \ref{thm:good}]
Similarly to the proof of Lemma \ref{lem:basic} we may equivalently express the
linearized Boussinesq equations around the affine temperature profile in Fourier variables as:
\begin{align}
  \label{eq:9}
      \dt
      \begin{pmatrix}
        \tilde{\omega}\\ k\tilde{\theta}
      \end{pmatrix}
      =
      \begin{pmatrix}
        -\nu (k^2+(\xi-kt)^2) & i \\
        \alpha \frac{i}{1+(\frac{\xi}{k}-t)^2} & 0
      \end{pmatrix}
                                           \begin{pmatrix}
                                            \tilde{\omega}\\ k\tilde{\theta} 
                                           \end{pmatrix},
    \end{align}
    where we consider coordinates $(k,\eta+kt)$ moving with Couette flow. Since the evolution of the $x$-averages of $\omega$ and $\theta$ decouples,
    in the following we without loss of generality only consider $k\neq 0$.
    
    We stress that the coefficients here are time-dependent and hence
this ODE system cannot anymore be explicitly solved in terms of a matrix
exponential.
However, a main advantage of the affine setting is that various estimates completely decouple, restrictions become trivial and
  operators commute, which makes this problem much simpler than the general
  profile case of Section \ref{sec:nonaffine} or the nonlinear problem of
  Section \ref{sec:nonlinear}.

  We note that the problem \eqref{eq:9} decouples with respect to $k$ and $\xi$,
  which we thus in the following treat as arbitrary but fixed.
  We then claim that for any $C>1$, $\nu>0$ and any $\alpha \in \R$ it holds that
  \begin{align}
    \label{eq:10}
    |\tilde{\omega}(t)|^2 + k^2 |\tilde{\theta}(t)|^2 \leq (1+\frac{1}{|\alpha|}) (1+C^2) \exp(\frac{|\alpha|}{\nu C^2}) (|\tilde{\omega}(0)|^2 + k^2 |\tilde{\theta}(0)|^2)
  \end{align}
  and thus the solution at time $t$ is controlled in terms of the initial data.
  We observe some special cases for the exponential:
\begin{itemize}
\item If we choose $C=1$ we obtain a bound by
  \begin{align}
    \exp(\frac{|\alpha|}{\nu}).
  \end{align}
  This bound holds for all $\alpha$, but suggests a threshold $|\alpha|<\nu$.
\item If we choose $C=\nu^{-1/4}$ we obtain a bound by
  \begin{align}
    \exp(|\alpha|),
  \end{align}
  where only the algebraic prefactors depends on $\nu$.
\item If we choose $C=\nu^{-1/3}$ we obtain a bound by
  \begin{align}
    \exp(|\alpha| \nu^{-1/3}),
  \end{align}
  where the exponent becomes uniformly bounded if we assume that $|\alpha|<\nu^{1/3}$.
\end{itemize}
The results of the theorem follow from the third case, where estimates in $H^N$
are obtained by integrating the frequency-wise bound \eqref{eq:10}.

  In order to introduce ideas and motivate the definition of $C$ we first discuss
  the case $\alpha>0$.
  
  \underline{Step 1 (symmetrize):} Let $\alpha>0$ be given. That is, suppose we
  are in the setting of hydrostatic balance.
  Then one commonly exploited feature in the setting without shear is
  cancellation of the purely imaginary off-diagonal entries (compare \cite{zillinger2020enhanced,doering2018long}).

  Indeed, consider the rescaled problem
  \begin{align}
    \dt
    \begin{pmatrix}
      \sqrt{\alpha} \tilde{\omega}\\ \sqrt{k^2+(\xi-kt)^2}\tilde{\theta} 
    \end{pmatrix}
    =
    \begin{pmatrix}
      -\nu (k^2+ (\xi-kt)^2) & \frac{i\sqrt{\alpha}}{\sqrt{1+(\frac{\xi}{k}-t)^2}} \\
      \frac{i\sqrt{\alpha}}{\sqrt{1+(\frac{\xi}{k}-t)^2}} & \frac{t- \frac{\xi}{k}}{1+(\frac{\xi}{k}-t)^2}
    \end{pmatrix}
               \begin{pmatrix}
      \sqrt{\alpha}\tilde{\omega}\\\sqrt{k^2+(\xi-kt)^2} \tilde{\theta} 
    \end{pmatrix}  
  \end{align}
  We observe that the off-diagonal entries are then exactly equal and imaginary
  and thus cancel under the matrix-valued map $M \mapsto M + \overline{M}^{T}$.
  
  Therefore, if we denote the square of the Euclidean norm of the vector as 
  \begin{align*}
      E(t):= \alpha |\tilde{\omega}|^2 + (k^2+(\xi-kt)^2)|\tilde{\theta}|^2
  \end{align*}
  it holds that 
  \begin{align*}
    \dt E(t)
    &= - \nu (k^2+(\xi-kt)^2)\alpha |\tilde{\omega}|^2 + \frac{t- \frac{\xi}{k}}{1+(\frac{\xi}{k}-t)^2}(k^2+(\xi-kt)^2)|\tilde{\theta}|^2\\
    &\leq \frac{\min(0,t- \frac{\xi}{k})}{1+(\frac{\xi}{k}-t)^2} E(t)
  \end{align*}
    Integrating in time and using that
    \begin{align*}
      \int_0^t \frac{\min(0,t- \frac{\xi}{k})}{1+(\frac{\xi}{k}-t)^2} \leq \ln (1+t^2)
    \end{align*}
    it follows that
    \begin{align}
      \label{eq:11}
      E(t)\leq (1+t^2)E(0).
    \end{align}
    Thus, irrespective of the size of $\alpha>0$ and of $\nu \geq 0$ we have shown that the
    evolution in $H^N$ is at most algebraically unstable.

    \underline{Step 2(Using dissipation):}
    Compared to our desired result, the estimate by \eqref{eq:11} is not yet
    sufficient, since it is not uniform in time.

    In the following we hence modify the definition of $E$ to also make use of
    the dissipation.
    More precisely, we introduce a cut-off
    \begin{align}
      C> 1
    \end{align}
    to be specified later and define the resonant time interval
    \begin{align}
      \label{eq:26}
    I:= \{t\geq 0 : |\frac{\xi}{k}-t|\leq C\}.
    \end{align}
    Then it holds that
    \begin{align}
      \label{eq:8}
      \begin{split}
      & \quad \dt
      \begin{pmatrix}
        \sqrt{\alpha} \tilde{\omega}\\ \sqrt{k^2+\min((\xi-kt)^2, C^2)}\tilde{\theta} 
      \end{pmatrix}
      \\
     & =
      \begin{pmatrix}
        -\nu (k^2+ (\xi-kt)^2) & \frac{ik\sqrt{\alpha}}{\sqrt{k^2+\min((\xi-kt)^2, C^2)}} \\
        \frac{ik\sqrt{\alpha}}{\sqrt{k^2+\min((\xi-kt)^2, C^2)}}\frac{\sqrt{k^2+\min((\xi-kt)^2, C^2)}}{\sqrt{1+(\frac{\xi}{k}-t)^2}} & \frac{t- \frac{\xi}{k}}{1+(\frac{\xi}{k}-t)^2} 1_{I}(t)
      \end{pmatrix}
                                                                                                                                      \begin{pmatrix}
      \sqrt{\alpha}\tilde{\omega}\\ \sqrt{k^2+\min((\xi-kt)^2, C^2)} \tilde{\theta} 
    \end{pmatrix},  
      \end{split}
  \end{align}
  where $1_I(t)\in \{0,1\}$ denotes the indicator function of $I$.
  We then define the modified energy as 
  \begin{align}
    \label{eq:12}
  E(t):= \alpha \|\omega\|_{H^N}^2 + \|\p_x\theta\|^2 + \|\min(\xi-kt, C)\tilde{\theta}\|_{L^2_N}^2
  \end{align}
  
\underline{Step 2a (resonant region):}
If $t \in I$ and $\alpha>0$ the problem and the definition of $E(t)$ are
identical to the one considered in Step 1 and it follows that
\begin{align}
  \label{eq:7}
  \dt E(t) \leq  \frac{\min(0,t- \frac{\xi}{k})}{1+(\frac{\xi}{k}-t)^2} E(t).
\end{align}
However, by definition of the interval $I$ it holds that
\begin{align*}
  \int_{I} \frac{\min(0, t - \frac{\xi}{k})}{1+(\frac{\xi}{k}-t)^2} dt \leq \ln (1+C^2)
\end{align*}
and thus the growth of $E$ during the resonant time is bounded by $(1+C^2)$.

\underline{Step 2b (non-resonant region):}
Next suppose that $t \not \in I$ and thus $|\frac{\xi}{k}-t|$ is large.
In particular, $(\xi-kt)^2 \leq k^2+ (\xi-kt)^2 \leq 2 (\xi-kt)^2$ and thus vertical dissipation
is comparable to full dissipation.

Then the off-diagonal entries in \eqref{eq:8} can be estimated as 
\begin{align*}
  \left|\frac{i\sqrt{\alpha}}{\sqrt{1+\min((\frac{\xi}{k}-t)^2, C^2)}} \right| &= \frac{\sqrt{\alpha}}{C^2}, \\
  \left| \frac{i\sqrt{\alpha}}{\sqrt{1+\min((\frac{\xi}{k}-t)^2, C^2)}}\frac{\sqrt{k^2+\min((\xi-kt)^2, C^2)}}{\sqrt{1+(\frac{\xi}{k}-t)^2}} \right| &\leq \frac{\sqrt{\alpha}}{C^2}.
\end{align*}
Thus, using Young's inequality with
\begin{align*}
  \frac{\sqrt{\nu} \sqrt{k^2+(\xi-kt)^2}}{\sqrt{\nu}\sqrt{k^2+(\xi-kt)^2}},
\end{align*}
we deduce that
\begin{align}
  \label{eq:6}
  \begin{split}
  \dt E(t)&\leq -\frac{\nu}{2} (k^2+(\xi-kt)^2) \alpha |\tilde{\omega}|^2 + \frac{\alpha}{\nu C^4} \frac{1}{k^2+(\xi-kt)^2} |\sqrt{k^2+\min((\xi-kt)^2, C^2)} \tilde{\theta}|^2 \\
  &\leq \frac{\alpha}{\nu C^4} \frac{1}{k^2+(\xi-kt)^2} E(t).
  \end{split}
\end{align}
We note that the factor on the right-hand-side is integrable in time.

\underline{Step 2c (Conclusion for $\alpha>0$)}
Combining the resonant estimate \eqref{eq:7} and the non-resonant estimate \eqref{eq:6}, we deduce that
\begin{align}
  \begin{split}
  \dt E(t)&\leq (1_{I}(t)\frac{\min(0,t- \frac{\xi}{k})}{1+(\frac{\xi}{k}-t)^2} + (1-1_{I}(t))\frac{\alpha}{\nu C^4} \frac{1}{k^2+(\xi-kt)^2}) E(t) \\
 \Rightarrow E(t)&\leq (1+C^2) \exp(\frac{\alpha}{\nu C^4}) E(0),
  \end{split}
\end{align}
with $E(t)$ defined in \eqref{eq:12}. The claimed estimate \eqref{eq:10} for $\alpha>0$ then follows by
comparing $E(t)$ with the squares of the $H^N$ norms.
It remains to discuss the case of negative $\alpha$.

\underline{Step 3 (negative $\alpha$)}
Let now $\alpha<0$ be given and consider the problem rescaled by
$\sqrt{|\alpha|}$ instead.
Then the evolution equation \eqref{eq:8} reads
\begin{align}
  \label{eq:13}
   & \quad  \dt
    \begin{pmatrix}
      \sqrt{|\alpha|} \tilde{\omega}\\ \sqrt{k^2+\min((\xi-kt)^2, C^2)}\tilde{\theta} 
    \end{pmatrix}
  \\
  &=
    \begin{pmatrix}
      -\nu (k^2+ (\xi-kt)^2) & \frac{i\sqrt{|\alpha|}}{\sqrt{1+\min((\frac{\xi}{k}-t)^2, C^2)}} \\
      - \frac{i\sqrt{|\alpha|}}{\sqrt{1+\min((\frac{\xi}{k}-t)^2, C^2)}}\frac{\sqrt{1+\min((\frac{\xi}{k}-t)^2, C^2)}}{\sqrt{1+(\frac{\xi}{k}-t)^2}} & \frac{t- \frac{\xi}{k}}{1+(\frac{\xi}{k}-t)^2} 1_{I}(t)
    \end{pmatrix}
               \begin{pmatrix}
      \sqrt{|\alpha|}\tilde{\omega}\\ \sqrt{k^2+\min((\xi-kt)^2, C^2)} \tilde{\theta} 
    \end{pmatrix}. 
  \end{align}
  We thus define the energy as
  \begin{align*}
    E(t)= |\alpha| |\tilde\omega|^2 + (k^2+\min((\xi-kt)^2, C^2)) |\tilde{\theta}|^2,
  \end{align*}
  which agrees with the previous definition if $\alpha>0$.

  \underline{Step 3a (non-resonant region):}
  Suppose that $t \not \in I$. We observe that in Step 2b we did not make use of
  the sign of $\alpha$ but only used Young's inequality.
  Furthermore, in that region $|\xi-kt|\geq |k|$ and thus in this region
  vertical dissipation dominates full dissipation.
  
  Hence, by the same argument we may deduce that also for our extended definition
  of $E(t)$ it holds that 
  \begin{align}
     \dt E(t) \leq \frac{\alpha}{\nu C^4} \frac{1}{k^2+(\xi-kt)^2} E(t).
  \end{align}

  \underline{Step 3b (resonant region):}
  Suppose that $t \in I$. Then we observe that off-diagonal terms in \eqref{eq:13} are of the
  same size but have the opposite sign an hence do not cancel anymore.
  However, we may use Young's inequality to still bound
  \begin{align*}
    \dt E(t) \leq -\nu (k^2+(\xi-kt)^2)|\alpha||\tilde{\omega}|^2 + \frac{|\alpha|}{\sqrt{1^2+(\frac{\xi}{k}-t)^2}} E(t)
    + \frac{t- \frac{\xi}{k}}{1+(\frac{\xi}{k}-t)^2} E(t) \\
    \leq \frac{1+|\alpha|}{\sqrt{1^2+(\frac{\xi}{k}-t)^2}} E(t),
  \end{align*}
  which yields a bound on the total growth by
  \begin{align*}
    (1+C^2)^{1+|\alpha|}.
  \end{align*}

  Combining the estimates in the resonant and non-resonant region, we deduce
  that 
  \begin{align}
    E(t)\leq (1+C^2)^{1+|\alpha|}\exp(\pi \frac{|\alpha|}{\nu C^4}) E(0).
  \end{align}
  In particular, choosing $C=\nu^{-1/3}$ and supposing that
  $|\alpha|<\min(\nu^{1/3}, 1)$, this estimate reduces to
  \begin{align*}
    E(t)\leq (1+\nu^{-2/3})^2 e^{\pi} E(0),
  \end{align*}
  which implies the result.

  We remark that in the proof for negative $\alpha$ we have not relied on
  cancellation but only on smallness of $\sqrt{|\alpha|}$ in combination with
  Young's inequality. Hence, we may consider a modification of the energy $E(t)$
  as
  \begin{align*}
    |\hat{\alpha}| |\tilde{\omega}|^2 + (k^2+ \min((\xi-kt)^2, C^2))|\tilde{\theta}|^2
  \end{align*}
  with $\hat{\alpha}=\max(|\alpha|, \nu^{1/3})$ and repeat the same proof, since
  $\frac{|\alpha|}{\sqrt{\hat{\alpha}}}< \nu^{1/6}$ and
  $\sqrt{\hat{\alpha}}<\nu^{1/6}$ satisfy the desired inequalities.
\end{proof}

We remark that in the proof of this affine case we can allow $\alpha$
to be arbitrarily large and are also free to choose $C$ arbitrarily.
As we discuss in the following, if $T'$ is non-constant or if we study the
nonlinear problem, smallness of $\alpha$ is required in the proof.
In view of resonances in the related linear inviscid damping problem
\cite{deng2019smallness} some form of smallness condition is probably necessary.

\subsection{Non-affine Temperature}
\label{sec:nonaffine}

Having discussed the setting of affine hydrostatic (im)balance, we next consider
$T(y)$ non-affine and address the question Q3 of Section \ref{sec:model} under
which conditions on $T$ in terms of $\nu$ such solutions are stable.
Here the problem does not decouple in frequency anymore and we thus employ a
by now classical Cauchy-Kowalewskaya or ghost energy approach (compare
\cite{Villani_long,bedrossian2015inviscid,Zill5}).

The linearized system around $\theta=T(y)$ in Lagrangian coordinates is given by: 
\begin{align}
  \label{eq:5}
  \begin{split}
  \dt \omega &= \nu (\p_y-t\p_x)^2 \omega + \p_x \theta, \\
  \dt \p_x \theta&= -T'(y)\p_x^2 \Delta^{-1}_t \omega,
  \end{split}
\end{align}
where we applied a derivative in $x$ to the second equation.
Since the evolution of the $x$-averages decouples, we assume without loss of
generality that
\begin{align}
  \int \omega dx = 0 =  \int \theta dx
\end{align}
throughout this section.

Our main results are summarized in the following theorem.
\begin{thm}
\label{thm:bad}
  Let $T(y)$ be a given temperature profile, $N \in \N$ and consider the linearized
  Boussinesq equations \eqref{eq:5} with vertical dissipation $\nu>0$.
  Further suppose that the Fourier transform of $T'$ satisfies
  \begin{align}
    \label{eq:18}
    \int (1+|\xi|)^{N+5} |\mathcal{F}(T')(\xi)|  \leq 4^{-N} \nu^{1/3},
  \end{align}
  then for any initial data $\omega_{0}, \theta_{0} \in H^N \times H^{N+1}$ it
  holds that
  \begin{align*}
    \|\omega(t)\|_{H^N}^2 + \|\p_x \theta(t)\|_{H^N}^2 \\
    \leq C (1+\nu^{-2/3})^2 ( \|\omega_0\|_{H^N}^2 + \|\theta_0\|_{H^{N+1}}^2).
  \end{align*}
\end{thm}

We remark that \eqref{eq:18} here is a sufficient condition to control several
commutators. We expect that in particular for large $N$ it is far from
sufficient and that it for instance would suffice to assume smallness for small
$N$ and only a finite norm for large $N$ (compare \cite{zillinger2019linear}).
Furthermore, if $T$ happens to be strictly increasing, stability is expected
also for large norms of $T'$. The main focus of this theorem thus lies on cases
where $T$ may be oscillating.
In the case $T(y)=\alpha y$, the (tempered) Fourier transform is given by a
Dirac measure and the condition \eqref{eq:18} reduces to $|\alpha|< \nu^{1/3}$, as in Section \ref{sec:linear}.

\begin{proof}
  In Section \ref{sec:linear} we had seen that in the special case when
  $T'(y)=\alpha y$ is affine, the functions $\theta, \omega$ satisfy the frequency-wise
  bound
  \begin{align}
    \label{eq:2}
    \begin{split}
    & \quad \dt (|\alpha| |\tilde{\omega}|^2 + (k^2+ \min((\xi-kt)^2, \nu^{-2/3})|\tilde{\theta}|^2) \\
    &\leq ( \frac{1}{\sqrt{1+(\frac{\xi}{k}-t)^2}}1_{I} + \frac{1}{1+(\frac{\xi}{k}-t)^2}) (|\alpha| |\tilde{\omega}|^2 + (k^2+ \min((\xi-kt)^2, \nu^{-2/3})|\tilde{\theta}|^2).
    \end{split}
  \end{align}

  Similarly to the (linear) inviscid damping problem in the Euler equations,
  while this frequency-wise bounds fail in the general setting, an integrated
  version can be shown to hold more generally.
  More precisely, we define two Fourier weights
  \begin{align}
    \label{eq:22}
    A(T,\xi, k) &= \exp(-2\int_0^T\frac{1}{1+(\frac{\xi}{k}-t)^2} dt), \\
    B(T,k,\xi) &= \exp(-2 \int_0^T\frac{1}{\sqrt{1+(\frac{\xi}{k}-t)^2}}1_{I} dt),
  \end{align}
  where we included a factor $2$ to have additional flexibility to absorb errors.

  Then in this affine case the estimate \eqref{eq:2} implies that if we define the energy
  \begin{align}
    \label{eq:14}
    E(t)= \alpha \|AB \omega\|_{H^N}^2 + \|AB \mathcal{F}^{-1} (k^2+ \min((\xi-kt)^2, \nu^{-2/3}) \mathcal{F} \theta\|_{H^N}^2,
  \end{align}
  where $\omega, \theta$ is a solution for $T(y)=\alpha y$, then $E(t)$ is
  non-increasing and moreover satisfies the decay estimate
  \begin{align*}
    \dt E(t)\leq - \iint ( \frac{1}{\sqrt{1+(\frac{\xi}{k}-t)^2}}1_{I} + \frac{1}{1+(\frac{\xi}{k}-t)^2}) (|\alpha| |AB \tilde{\omega}|^2 + (k^2+ \min((\xi-kt)^2, \nu^{-2/3})|AB \tilde{\theta}|^2).
  \end{align*}
  In particular, $E(t)$ is non-increasing and the inequality $E(t)\leq E(0)$
  implies the result of the theorem for the special case when $T$ is affine.\\

  Let now $T(y)$ be given and for simplicity of notation define
  \begin{align}
    \hl = \mathcal{F}^{-1}\sqrt{k^2+\min((\xi-kt)^2, \nu^{-2/3})}\mathcal{F}. 
  \end{align}
  and introduce the constant $\alpha$ in terms of operator norms:
  \begin{align}
    \alpha:= \||\nabla_t|^{-1} B T' B^{-1} |\nabla_t|\|_{H^{N} \mapsto H^{N}} + \|(\p_x^2 \Delta_t^{-1})^{-3/2} B T'(y) B^{-1} B (\p_x^2 \Delta_t^{-1})^{3/2}\|_{H^N \rightarrow H^N}.
  \end{align}
  As the last step of this proof we will show that by \eqref{eq:18} it follows
  that $\alpha\leq \nu^{1/3}$. Similarly as in Theorem \ref{thm:good}, we remark
  that in all the following estimates we may replace $\alpha$ by
  $\hat{\alpha}=\max(\alpha, \nu^{1/3})$ if $\alpha< \nu^{1/3}$.
  
  We now claim that if $E(t)$ is defined by the same formula as in \eqref{eq:14}
  but with $\omega,\theta$ being solutions of the linearized problem with
  temperature profile $T$, then $E(t)$ is non-increasing. This then implies the
  desired estimate by controlling $B$ and $\alpha$ (or $\hat{\alpha}$) in terms
  of $\nu$.

  We hence have to estimate
  \begin{align*}
    \dt E(t)/2 = - \nu \alpha \|\nabla_t AB \omega\|_{H^N}^2 + \alpha \langle  AB \omega, AB \p_x \theta \rangle \\
    + \langle AB \hl \theta, AB T'(y) B^{-1} B \p_x^2 \Delta_t^{-1} \omega \rangle \\
    + \alpha \langle (\dot{A} B + A\dot{B})\omega, AB \omega \rangle\\
    + \langle (\dot{A} B + A\dot{B}) \hl \theta, AB \hl \theta  \rangle \\
    + \langle AB \dot{\hl} \theta,  AB \hl \theta  \rangle \rangle.
  \end{align*}
  Here the dissipation terms and derivatives of $AB$ yield non-negative
  contributions and are thus beneficial and $\dot{B}$ was defined in such a way to
  control
  \begin{align*}
     \langle AB \dot{\hl} \theta,  AB \hl \theta  \rangle.
  \end{align*}
  More precisely, we note that inside the resonant interval $I$,
  \begin{align*}
    \dt \hl^2 = - 2k (\xi-kt) = \frac{-2 (\frac{\xi}{k}-t)}{1+(\frac{\xi}{k-t})^2} \hl^2
  \end{align*}
  can be controlled by $\dot B B$.
  
  It thus remains to estimate
  \begin{align}
    \label{eq:16}
    E_{\omega}:= \alpha \langle  AB \omega, AB \p_x \theta \rangle
  \end{align}
  and
  \begin{align}
    \label{eq:17}
    E_{\theta}:= \langle AB \hl \theta, AB T'(y) B^{-1} B \p_x^2 \Delta_t^{-1} \omega \rangle
  \end{align}

  \underline{Estimating $E_{\omega}$:}  
  Since the evolution equation for $\omega$ does not involve $T'(y)$ we may
  argue as in the affine case and control $E_{\theta}$ frequency-wise.
  More precisely, for any given frequency $(k,\xi)$ we need to control
  \begin{align*}
    \left| \alpha (AB)^2(t,k,\xi) \overline{\tilde{\omega}}(t,k,\xi) ik \tilde{\theta}(t,k,\xi) \right| 
  \end{align*}

  \underline{Resonant region:}
  If $t,k,\xi$ are such that $|\frac{\xi}{k}-t|\leq \nu^{-1/3}$ we may bound by
  this by
  \begin{align*}
    \frac{\sqrt{\alpha}}{\sqrt{1+(\frac{\xi}{k}-t)}} (AB)^2(t,k,\xi) \left[ \alpha |\tilde{\omega}|^2 + \left|\sqrt{1+(\frac{\xi}{k}-t)^2} \tilde{\theta}\right|^2 \right],
  \end{align*}
  which can be absorbed into
  \begin{align*}
    \langle A \dot{B} \omega, AB \omega \rangle +
        \langle A\dot{B} \hl \theta, AB \hl \theta  \rangle
  \end{align*}
  by construction of $B$.
  
  \underline{Non-resonant region:}
  If instead $t,k,\xi$ are such that $|\frac{\xi}{k}-t|\geq \nu^{-1/3}$, then
  $\dot B$ vanishes and we instead make use of the vertical dissipation.
  That is, we estimate
  \begin{align*}
  & \quad  \left| \alpha (AB)^2(t,k,\xi) \overline{\tilde{\omega}}(t,k,\xi) ik \tilde{\theta}(t,k,\xi) \right|
    \\ &= \alpha (AB)^2|\tilde{\omega}| \frac{\sqrt{\nu}\sqrt{k^2+(\xi-kt)^2}}{\sqrt{\nu}\sqrt{k^2+(\xi-kt)^2}} \frac{1}{\sqrt{1+(\frac{\xi}{k}-t)^2}} \sqrt{k^2+\nu^{-2/3}} |\tilde{\theta}|
  \end{align*}
  by the dissipation term
  \begin{align*}
    - \alpha (AB)^2 (\sqrt{\nu}\sqrt{(\xi-kt)^2}|\tilde{\omega}|)^2
  \end{align*}
  and
  \begin{align*}
   (AB)^2 \frac{1}{1+(\frac{\xi}{k}-t)^2} |\sqrt{k^2+\nu^{-2/3}} |\tilde{\theta}|^2. 
  \end{align*}
  Here we used that in the non-resonant region $(\xi-kt)^2$ controls the full
  dissipation.
  
  The latter term can then be absorbed into
  \begin{align*}
    \langle \dot{A}B \hl \theta, AB \hl \theta \rangle
  \end{align*}
  provided
  \begin{align}
    \label{eq:15}
    \frac{\sqrt{\alpha} }{\sqrt{\nu}\sqrt{k^2+(\xi-kt)^2}}
  \end{align}
  is less than $1$.
  Since we are in the non-resonant region \eqref{eq:15} can be bounded from
  above by
  \begin{align*}
    \sqrt{\alpha} \nu^{-1/2+ 1/3} = (\alpha \nu^{-1/3})^{1/2},
  \end{align*}
  which is small since $\alpha < \nu^{1/3}$ by assumption.

  \underline{Estimating $E_{\theta}$:}
  In order to estimate the contribution \eqref{eq:17}
  \begin{align*}
    E_{\theta}= \langle AB \hl \theta, AB T'(y) B^{-1} B \p_x^2 \Delta_t^{-1} \omega \rangle
  \end{align*}
  we follow a similar argument as in the affine case.
  However, as $T'$ is non-constant we further have to control an interaction
  term between the resonant and non-resonant regions.

  More precisely, for any given time $t$ we define the Fourier set
  \begin{align*}
    \Omega(t)=\{(k,\xi): k\neq 0 , |\frac{\xi}{k}-t|\leq \nu^{-1/3}\}.
  \end{align*}
  That is, instead of time interval $I$ associated to given frequencies, we
  consider frequencies for a given time $t$.
  We then split
  \begin{align*}
    \omega &= 1_{\Omega}\omega + (1-1_{\Omega}) \omega =: \omega_{\text{in}} + \omega_{\text{out}}, \\
    \theta &= 1_{\Omega}\theta + (1-1_{\Omega}) \theta =: \theta_{\text{in}} + \theta_{\text{out}}.
  \end{align*}
  We then split the contributions as
  \begin{align*}
    \langle AB \hl \theta , AB T'(y) B^{-1} B \p_x^2 \Delta_t^{-1} \omega_{\text{out}} \rangle \\
    \langle AB \hl \theta_{\text{in}}, AB T'(y) B^{-1} B \p_x^2 \Delta_t^{-1} \omega_{\text{in}} \rangle \\
    \langle AB \hl \theta_{\text{out}}, AB T'(y) B^{-1} B \p_x^2 \Delta_t^{-1} \omega_{\text{in}} \rangle.
  \end{align*}
  We remark that in the affine case the third term identically vanished due to
  the disjoint Fourier support of $\omega_{\text{in}}$ and $\theta_{\text{out}}$, but that
  this orthogonality is lost in the general case.

  \underline{Step 2a ($\omega_{\text{out}}$):}
  We argue as in the affine case. Since $\omega_{\text{out}}$ is supported in $\Omega$
  it holds that
  \begin{align*}
    \|\nabla_t B \p_x^2 \Delta_t^{-1} \omega_{\text{out}} \|_{H^N} \leq \nu^{2/3}\|(\p_y-t\p_x) B \omega \|_{H^N}.
  \end{align*}
  For $\theta$ we do not need a further control of the support and may bound
  \begin{align*}
    \||\nabla_t|^{-1} AB \hl \theta\|_{H^N}
  \end{align*}
  by the time decay of $A$, provided
  \begin{align*}
    \||\nabla_t|^{-1} B T'(y) B^{-1} |\nabla_t|\|_{H^N \rightarrow H^N}\leq \sqrt{\alpha} \nu^{1/6}.
  \end{align*}
  By our choice of $\alpha$ the left-hand-side is bounded by $\alpha$ and this
  estimate is therefore satisfied provided $\alpha <\nu^{1/3}$, as assumed.
  
  \underline{Step 2b ($\theta_{\text{in}}, \omega_{\text{in}}$):}
  Similarly as in the proof of Theorem \ref{thm:good} we use the time decay of
  $B$ to control this contribution.
  More precisely, we may bound this contribution in terms of
  \begin{align*}
   \alpha \langle AB \omega_{\text{in}} ,\sqrt{\p_x^2 \Delta_{t}^{-1}} AB \omega_{\text{in}} \rangle\\
   + \langle AB \hl \theta_{\text{in}} ,\sqrt{\p_x^2 \Delta_{t}^{-1}} AB \hl \theta_{\text{in}} \rangle,
  \end{align*}
  provided
  \begin{align*}
    \|\sqrt{\p_x^2 \Delta_{t}^{-1}} B T'(y) B^{-1}\sqrt{\p_x^2 \Delta_{t}^{-1}}^{-1} \|_{H^N \rightarrow H^N} \leq \sqrt{\alpha}.
  \end{align*}
  We remark that $\sqrt{\p_x^2 \Delta_{t}^{-1}}=|\p_x||\nabla_t|^{-1}$ and that
  $BT'(y)B^{-1}$ does not depend on $x$. Hence this estimate is equivalent to
  the one of step 2a.

  \underline{Step 2c ($\theta_{\text{out}}, \omega_{\text{in}}$)}
  As $T'$ is non-constant the contribution
  \begin{align*}
    \langle AB \hl \theta_{\text{out}}, AB T'(y) B^{-1} B \p_x^2 \Delta_t^{-1} \omega_{\text{in}} \rangle
  \end{align*}
  generally does not vanish.
  However, since $\theta_{\text{out}}$ is supported away from the resonant region, we
  may insert an identity operator $(\p_x^2 \Delta_t^{-1})^{3/2-3/2}$ and bound
  \begin{align*}
    \|AB (\p_x^2 \Delta_t^{-1})^{3/2} \hl \theta_{\text{out}}\|_{H^N} \leq \nu^{2/3} \sqrt{- \langle\dot{A}B \hl \theta_{\text{out}}, AB \hl \theta_{\text{out}}\rangle}
  \end{align*}
  and estimate 
  \begin{align*}
   \| (\p_x^2 \Delta_t^{-1})^{-3/2} AB T'(y) B^{-1} B \p_x^2 \Delta_t^{-1} \omega_{\text{in}}\|_{H^N}
  \end{align*}
  by
  \begin{align*}
    \|(\p_y-t\p_x) AB \omega\|_{H^N}.
  \end{align*}
  This contribution can thus be absorbed by the same argument as in Step 2a,
  provided
  \begin{align*}
    \|(\p_x^2 \Delta_t^{-1})^{-3/2} B T'(y) B^{-1} B (\p_x^2 \Delta_t^{-1})^{3/2}\|_{H^N \rightarrow H^N}\leq \nu^{1/6}\sqrt{\alpha},
  \end{align*}
  which by our definition of $\alpha$ reduces to $\alpha< \nu^{1/3}$.

  \underline{Step 4 (controlling $\alpha$):}
  It remains to be shown that the estimate \eqref{eq:18} controls $\alpha$.
  Here we make use of Schur's test, which controls the $L^2$ operator norm of a
  map
  \begin{align*}
    u(x)\mapsto \int K(x,y) u(y) dy
  \end{align*}
  by the square root of
  \begin{align*}
    \sup_{x}\int |K(x,y)| dy \sup_{y}\int |K(x,y)| dy.
  \end{align*}
  More precisely, we may express the map $u \mapsto |\nabla_t| BT'
B^{-1} |\nabla_t|^{-1} u$ as integration against a kernel on the Fourier side:
\begin{align*}
  \tilde{u}(k,\xi) \mapsto \int \sqrt{k^2+(\xi-kt)^2} B(t,k,\xi) \tilde{T'}(\xi-\zeta) B^{-1}(t,k,\zeta) \sqrt{k^2+(\zeta-kt)^2} \tilde{u}(k,\zeta) d\zeta.
\end{align*}
Since we are further interested in a map on $H^N$ we add an additional weight
\begin{align*}
  \frac{1+|\xi|^{N}}{1+|\zeta|^N}.
\end{align*}
Then Schur's test asks us to control
\begin{align*}
  \sup_{\xi}\int \sqrt{k^2+(\xi-kt)^2} B(t,k,\xi) \tilde{T'}(\xi-\zeta) B^{-1}(t,k,\zeta) \sqrt{k^2+(\zeta-kt)^2}\frac{1+|\xi|^{N}}{1+|\zeta|^N} d\zeta \leq C_1
\end{align*}
and
\begin{align*}
  \sup_{\zeta}\int \sqrt{k^2+(\xi-kt)^2} B(t,k,\xi) |\tilde{T'}(\xi-\zeta)| B^{-1}(t,k,\zeta) \sqrt{k^2+(\zeta-kt)^2}\frac{1+|\xi|^{N}}{1+|\zeta|^N} d\xi \leq C_2,
\end{align*}
which then bounds the $L^2$ operator norm by $\sqrt{C_1C_2}$.

We claim that this kernel can be bounded by $|\tilde{T'}(\xi-\zeta)|
(1+|\xi-\zeta|)^{N+5}$, at which point \eqref{eq:18} implies that
$C_1=C_2=\nu^{1/3}$, which concludes the proof.

Indeed, by construction of $B$, we can control
\begin{align*}
  B(t,k,\xi)B^{-1}(t,k,\zeta) \leq \sqrt{1+|\xi-\zeta|^2}.
\end{align*}

Similarly, if $|\xi|\leq 3 |\zeta|$, we may simply control
\begin{align*}
  \frac{1+|\xi|^{N}}{1+|\zeta|^N} \leq 1+3^N.
\end{align*}
If instead $|\xi|\geq 3 |\zeta|$, then
\begin{align*}
|\xi| \leq |\xi-\zeta| + \frac{1}{3}|\xi| \Leftrightarrow |\xi| \leq \frac{3}{2} |\xi-\zeta|.
\end{align*}
and thus
\begin{align*}
  \frac{1+|\xi|^{N}}{1+|\zeta|^N} \leq (\frac{3}{2})^N (1+|\xi-\zeta|^{N}).
\end{align*}

Finally, we need to control
\begin{align*}
  \frac{k^2+(\xi-kt)^2}{k^2+(\zeta-kt)^2} = \frac{1+ (\frac{\xi}{k}-t)^2}{1+(\frac{\zeta}{k}-t)^2}.
\end{align*}
Here we may simply estimate
\begin{align*}
  (\frac{\xi}{k}-t)^2 \leq 2 (\frac{\zeta}{k}-t)^2 +2 (\xi-\zeta)^2.
\end{align*}
The first term cancels with the numerator, while for the second we simply bound
by $|\xi-\zeta|$.

Thus, in total it suffices to bound
\begin{align*}
  \sup_{\zeta} \int |\tilde{T'}(\xi-\zeta)| (1+|\xi-\zeta|)^{N+5} d\xi <C_1=C_2,
\end{align*}
which is the assumption of our theorem.

\end{proof}

We remark that in the case when $T'$ is increasing stronger results are
possible, for instance allowing $\alpha$ to be much larger, by using additional
cancellations as in Section \ref{sec:linear}.
The main advantage of this theorem hence lies in the fact that we can allow $T'$ to
be decreasing or oscillating.

\section{The Nonlinear Equations with Vertical Dissipation}
\label{sec:nonlinear}
Given the results for the linearized problem, it is natural to ask whether they
extend to the nonlinear perturbed problem:
\begin{align*}
  \dt \omega + y \p_x \omega + v \cdot \nabla \omega &= \nu \p_y^2 \omega + \p_x \theta , \\
  \dt \theta + y \p_x \theta + T'(y) v_2 + v \cdot \nabla \theta &=0,\\
  (t,x,y)&\in (0,\infty)\times \T \times \R,
\end{align*}
and, if so, how this depends on $\nu$.
As shown recently by Masmoudi, Said-Houari and Zhao
\cite{masmoudi2020stability}, this problem may exhibit an instability reminiscent
of echo chains in the Vlasov-Poisson equations \cite{bedrossian2016nonlinear,zillinger2020landau} and Euler equations \cite{dengmasmoudi2018,dengZ2019}.
For this reason, we do not expect results in Sobolev regularity to extend
(without strong modification).
Therefore, in this section we instead consider the more viscous problem
\begin{align}
  \label{eq:19}
  \begin{split}
  \dt \omega + y \p_x \omega + v \cdot \nabla \omega &= \nu \p_y^2 \omega + \p_x \theta , \\
  \dt \theta + y \p_x \theta + T'(y) v_2 + v \cdot \nabla \theta &= \nu \p_y^2 \theta,\\
  (t,x,y)&\in (0,\infty)\times \T \times \R,
  \end{split}
\end{align}
where we impose \emph{full vertical dissipation} and view $T(y)$ as a solution of the
forced problem. 
Similarly to results for the case of hydrostatic balance with shear studied in
\cite{zillinger2020enhanced} our aim here is to extend the linear (asymptotic)
stability results to the nonlinear equations with small data and thus answer
question Q4 of Section \ref{sec:model}.

\begin{thm}
  Let $N\geq 5$ and suppose that the temperature profile $T(y)$ satisfies the
  linear stability assumptions of Theorem \ref{thm:bad}.
  Let further $0< \epsilon< \nu^2$ and suppose that the initial data satisfies
  \begin{align*}
    \|\omega_0\|_{H^N} + \nu^{-1/2}\|\p_x \theta_0\|_{H^N} \leq \epsilon. 
  \end{align*}
  The the unique global solution with this initial data satisfies
  \begin{align*}
    \|\omega\|_{L^\infty H^N} + \nu \|(\p_y-t\p_x) \omega  \|_{L^2 H^N} + \|v_{\neq}\|_{L^2 H^N} &\leq 10 \nu^{-1/3}\epsilon, \\
     \|\p_x \theta\|_{L^\infty H^N} &\leq 10 \epsilon,
  \end{align*}
  where $L^p H^N := L^p((0,\infty); H^N)$ and $v_{\neq}=v- \int v dx$ denotes
  the non-shear component of the velocity.
\end{thm}
\begin{rem}
  \begin{itemize}
    \item The nonlinear problem with vertical dissipation but without shear has been previously
studied in \cite{cao2013global,li2016global} and \cite{adhikari20102d}.
\item The threshold $\epsilon < \nu^2$ here is imposed to control losses of powers
  $\nu^{1/3}$ in enhanced dissipation estimates encoded in our Fourier
  multiplier $B$. 
\item The nonlinear problem without thermal dissipation has been recently
  studied in \cite{masmoudi2020stability}. In particular, they require Gevrey
  regularity to control resonances, which suggests that stability in Sobolev
  regularity may either fail or require non-trivial modification
  \cite{dengZ2019,dengmasmoudi2018}.
\item In a previous work \cite{zillinger2020enhanced} we studied the special
  case where $T(y)$ is affine (with positive slope) and with full dissipation. The present result allows for possibly oscillating
  profiles and only requires vertical dissipation.
\item We remark that we here estimate $\hl \theta$ instead of $\nabla_t \theta$
  or $\p_x \theta$. This is in view to the results of
  Section \ref{sec:linear}, for which we do not expect control of $\nabla_t \theta$.
\item In view of the partial dissipation results of Section
    \ref{sec:hydroimbalance} we here omit questions of enhanced dissipation.
  \item There has been extensive work on various partial dissipation regimes as
    well as on the inviscid problem. We discuss some of this literature
    in the introductory Section \ref{sec:introduction}.
  \end{itemize}
\end{rem}

\begin{proof} 
We follow a classical bootstrap argument approach \cite{Villani_long,bedrossian2016sobolev,liss2020sobolev} in the spirit of
Cauchy-Kowalewskaya.
As in \cite{zillinger2020enhanced} we here make use of multipliers constructed in
\cite{bedrossian2016sobolev} and \cite{liss2020sobolev} for the Navier-Stokes
and MHD problems, respectively, and adapt them to the problem at hand.
In contrast to these works we do not aim to derive (enhanced)
dissipation estimates.
However, we show that vertical dissipation is sufficient to employ these
bootstrap methods
(see also the discussion of echo chains \cite{masmoudi2020stability} in
Section \ref{sec:introduction}).
We remark that in Section \ref{sec:nonaffine} we have derived estimates for the
associated linearized problem, which we use as a basis for
our estimates in the following.
A main challenge in the control of various contributions here will be that we
can only control vertical dissipation and hence will have to separately consider
regimes where horizontal dissipation would be large.\\

In our bootstrap construction we consider $L^p H^N$ norms on a time interval
$(0,T)$, $T>0$, which incorporate a time-dependent Fourier multiplier $M$ with
\begin{align*}
  \nu^{1/3}\leq M \leq 1,
\end{align*}
to be specified later (see equation \eqref{eq:25}).
We then consider the maximal time $T>0$ such that the following \emph{bootstrap
estimates} are satisfied: 
  \begin{align}
    \label{eq:21}
    \begin{split}
    \begin{split}
     &\|M \omega_{\neq}\|_{L^\infty_t H^N}^2
    + \nu \|(\p_y-t\p_x) M \omega_{\neq}  \|_{L^2 H^N}^2 +
    \nu \|1_{|\xi-kt|\leq |k|} M \omega_{\neq}\|_{L^2 H^N}^2\\
   & \quad +  
    \|\nabla_t \Delta_t^{-1} M \omega_{\neq} \|_{L^2 H^N}^2 \leq 16 \epsilon^2, 
  \end{split}\\
    \begin{split}
      & \|\hl M \theta_{\neq}\|_{L^\infty H^N}^2 + \nu \|(\p_y-t\p_x) \hl M \theta_{\neq} \|_{L^2 H^N}^2 + \nu \|1_{|\xi-kt|\leq |k|} \hl M \theta_{\neq}\|_{L^2 H^N}^2\\
      & \quad + \|\nabla_t \Delta_t^{-1}\hl M \theta_{\neq} \|_{L^2 H^N}^2 \leq 16 \nu \epsilon ^2, 
    \end{split}\\
    \begin{split}
    \|\omega_{=}\|_{L^\infty_t H^N}^2 + \nu \|\p_y \omega_{=}\|_{L^2 H^N}^2\leq 16\epsilon^2 , 
  \end{split}\\
    \begin{split}
    \|\hl \theta_{=}\|_{L^\infty H^N}^2+ \nu \|\hl \p_y \omega_{=}\|_{L^2 H^N}^2 \leq 16 \nu \epsilon^2, 
    \end{split}
    \end{split}
  \end{align}
  where $\omega_{=}, \theta_{=}$ denote the $x$-averages and $\omega_{\neq},
  \theta_{\neq}$ their orthogonal complement and $\hl$ is the Fourier multiplier
  \begin{align*}
    \hl =\mathcal{F}^{-1}(k^2+(\min(\xi-kt)^2,\nu^{-2/3}))\mathcal{F}.
  \end{align*}

By local well-posedness and the assumed existence of a solution, there exists
some positive time $T>0$ such that \eqref{eq:21} holds with $L^2(\R_{+}; \cdot)$
and $L^\infty(\R_{+}; \cdot)$ replaced by $L^2((0,T);\cdot)$ and
$L^{\infty}((0,T); \cdot)$.
If the maximal time $T$ with this property is infinity, this yields the results
of the theorem in view of the bounds on $M$.

In the following we thus assume for the sake of contradiction that $T<\infty$ is maximal.
We will then show that at the time $T$ none of the estimates in
\eqref{eq:21} attain equality. Therefore, by continuity the estimates are still
satisfied for a slightly larger time, which contradicts the maximality and thus
implies the result.

In order to introduce ideas let us first consider the $x$-averages. We remark
that in the linearized results of Section \ref{sec:hydroimbalance} their
evolution decoupled and reduced to heat evolution. Thus, in the following we
have to control the effects of the nonlinearity, where the lack of full
dissipation requires us to introduce some additional splittings.\\

\underline{Estimating $\omega_{=}$:}
We observe that $\p_x \theta$, $v_{\neq}\cdot \nabla_t\omega_{=}$ and
$v_{=}\cdot \nabla_t\omega_{\neq}$ all posses a vanishing $x$-average and thus
obtain the following evolution equation for $\omega_{=}$:
\begin{align*}
  \dt \omega_{=} +0 + (v_{\neq} \cdot \nabla_t \omega_{\neq})_{=}= \nu \p_y^2 \omega_{=} + 0.
\end{align*}
Testing this equation with $\omega_{=}$ and integrating in time, we deduce that
\begin{align}
  \|\omega_{=}(T)\|_{H^N}^2 + 2\nu \int_{0}^T \|\p_y \omega_{=}\|_{H^N}^2 = \|\omega_{=}(0)\|_{H^N}^2 + \int_0^T \langle \omega_{=},v_{\neq} \cdot \nabla_t \omega_{\neq}  \rangle.
\end{align}
We recall that by assumption the initial data is of size much smaller than
$\sqrt{8}\epsilon$. Thus, if we can show that the integral on the
right-hand-side is bounded by $\epsilon$, this implies that equality in
\eqref{eq:21} is indeed not attained here.

As we assume only vertical dissipation, we first discuss the part involving $y$
derivatives of $\omega_{\neq}$:
\begin{align*}
  \int_0^T \langle \omega_{=},v_{\neq}^{2}  (\p_{y}-t\p_x) \omega_{\neq}  \rangle \\
  \leq \|\omega_{=}\|_{L^\infty H^N} \|v_{\neq}^2\|_{L^2 H^N} \|(\p_{y}-t\p_x) \omega_{\neq}\|_{L^2 H^N} \\
  \leq 4\epsilon  4\nu^{-1/3}\epsilon \nu^{-1/2} 4\nu^{-1/3} \epsilon = 4\nu^{-7/6}\epsilon 16 \epsilon^2,
\end{align*}
where $v_{\neq}^2$ denotes the vertical component of the velocity field, and the
loss of factors $\nu^{-1/3}$ is due to the multiplier $M$.
Since by assumption $4\nu^{-7/6}\epsilon$ is much smaller than $1$, this
term is too small to help achieve equality.

For the term involving $x$-derivatives, we introduce a Fourier multiplier $\chi$
which corresponds to the projection onto the set
\begin{align*}
  \{ (k,\xi): |\xi-kt|\geq |k| \}.
\end{align*}
Then by construction it holds that
\begin{align*}
  \|\chi \p_x \omega_{\neq}\|_{L^2 H^N} \leq \|\chi (\p_{y}-t\p_x) \omega_{\neq}\|_{L^2 H^N} \leq \|(\p_{y}-t\p_x) \omega_{\neq}\|_{L^2 H^N},
\end{align*}
which thus allows for an estimate of the same form as for the part involving $y$
derivatives.

Finally, we estimate
\begin{align*}
  \int_0^T \langle \omega_{=},v_{\neq}^{1}  \p_x (1-\chi) \omega_{\neq}  \rangle\\
  = -\int_0^T  \langle \omega_{=},(\p_x v_{\neq}^{1})  (1-\chi) \omega_{\neq}  \rangle \\
  \leq \|\omega_{=}\|_{L^\infty H^N} \|v_{\neq}^1\|_{L^2 H^N} \| (1-\chi) \omega_{\neq}\|_{L^2 H^N} \\
  \leq  4\epsilon \nu^{-7/6} 4\epsilon 4\epsilon,
\end{align*}
where we lose several powers of $\nu$ due the enhanced multiplier
$B$ discussed in Section \ref{sec:nonaffine}.
As this contribution is also much smaller than $16 \epsilon^2$, we conclude that
\begin{align*}
  \|\omega_{=}(T)\|_{H^N}^2 + \nu \int_{0}^T \|\p_y \omega_{=}\|_{H^N}^2 < 8 \epsilon^2 
\end{align*}
and thus equality in \eqref{eq:21} is not attained.
\\

\underline{Estimating $\theta_{=}$:}
Before discussing $\hl \theta_{=}$, we consider $\theta_{=}$, where we can argue
analogously to the case of $\omega_{=}$.
We may test the equation
\begin{align*}
  \dt \theta_{=} + (v_{\neq}\cdot \nabla_t \theta_{\neq})_{=} = \nu \p_y^2 \theta_{=}
\end{align*}
with $\theta_{=}$ and integrate in time to again derive an integral estimate.
We then estimate the contribution
\begin{align*}
  \int_0^T \langle \theta_{=}, v_{\neq}\cdot \nabla_t \theta_{\neq} \rangle
\end{align*}
by
\begin{align*}
  \|\theta_{=}\|_{L^\infty H^N} \|v_{\neq}^2\|_{L^2 H^N} \|(\p_y-t\p_x) \theta_{\neq}\|_{L^2 H^N}\\
  + \|\theta_{=}\|_{L^\infty H^N} \|v_{\neq}^1\|_{L^2 H^N} \|(\p_y-t\p_x) \chi \theta_{\neq}\|_{L^2 H^N}\\
  + \|\theta_{=}\|_{L^\infty H^N} \|\p_x v_{\neq}^1\|_{L^2 H^N} \| (1-\chi) \theta_{\neq}\|_{L^2 H^N}.
\end{align*}
By the bootstrap assumptions this sum can be controlled in terms of $\nu^{-7/6}
\epsilon^3$, which is much smaller than $\epsilon^2$.\\

\underline{Estimating $\hl \theta_{=}$:}
We may extend the definition of $\hl$ to purely $y$-dependent functions as the
Fourier multiplier $\hl = \mathcal{F}^{-1}\min(|\xi|, \nu^{-1/3}) \mathcal{F}$.
We note that the operator norm of $\hl$ is bounded by $\nu^{-1/3}$ and thus $\hl
\theta_{=}$ could be controlled in terms of $\theta_{=}$. However, in this way
we would pass from a bound by $\epsilon^2$ to one by $\nu^{-2/3}\epsilon^2$,
which is insufficient for our bootstrap approach.
Instead we aim to show that by a similar argument as above $\|\hl
\theta_{=}\|_{L^\infty H^N}$ can be controlled by $\epsilon^2$, where the loss
of powers of $\nu$ only factors into the smallness conditions on $\epsilon$ used
to control nonlinear interaction terms.

We may control
\begin{align*}
  & \quad \|\hl \theta_{=}(T)\|_{H^N}^2 +\nu \|\p_y \hl\theta_{=}\|_{L^2 H^N}^2 \\
  &= \|\hl \theta_{=}(0)\|_{H^N}^2 + \int_0^T \langle \hl \theta_{=}, \hl v_{\neq}\cdot \nabla_t \theta_{\neq} \rangle \\
  &\leq \|\hl \theta_{=}(0)\|_{H^N}^2 + \|\hl \theta_{=}\|_{L^\infty H^N} \nu^{-1/3} (\|v_{\neq}^2\|_{L^2 H^N} \|(\p_y-t\p_x) \theta_{\neq}\|_{L^2 H^N} \\
  &\quad +\|v_{\neq}^1\|_{L^2 H^N} \|(\p_y-t\p_x) \chi \theta_{\neq}\|_{L^2 H^N} + \|\p_x v_{\neq}^1\|_{L^2 H^N} \| (1-\chi) \theta_{\neq}\|_{L^2 H^N}).
\end{align*}
Thus, by assumption on $\epsilon$ and the initial data, equality in
\eqref{eq:21} is also not achieved for $\hl \theta_{=}$. \\

\underline{Estimating $\omega_{\neq}$ and $\theta_{\neq}$:}
Having discussed the control of the $x$-averages, we now turn to control
$\omega_{\neq}, \hl \theta_{\neq}$. Here we will first focus on contributions
due to $T(y)$ and the $x$-averages and finally discuss the control of the
nonlinearity involving $v_{\neq}$.

We recall that $\omega_{\neq}$ and $\theta_{\neq}$ satisfy the system
\begin{align*}
  \dt \omega_{\neq} &= \nu (\p_y-t\p_x)^2 \omega_{\neq} + \p_x \theta_{\neq} - v_{=}^1\p_x \omega_{\neq} - v_{\neq}^2 \p_y \omega_{=} - (v_{\neq}\cdot \nabla_t \omega_{\neq})_{\neq}, \\
   \dt \theta_{\neq} &= \nu (\p_y-t\p_x)^2 \omega_{\neq} + T'(y) v_{\neq}^2  - v_{=}^1 \p_x \theta_{\neq} - v_{\neq}^2 \p_y \theta_{=}- (v_{\neq}\cdot \nabla_t \theta_{\neq})_{\neq},
\end{align*}
where we consider $\omega_{=}$ and $\theta_{=}$ as given functions.

In the linearized problem of Section \ref{sec:nonaffine} we could without loss
of generality assume that $\omega_{=}=\theta_{=}=0$ and constructed a non-increasing
energy functional.
In the following we build on these estimates and integrate them in time to show that the control \eqref{eq:21} is stable under small nonlinear perturbations.

We recall the multipliers $A,B$ defined in \eqref{eq:22} in Section \ref{sec:nonaffine}:
\begin{align*}
  A(t,k,\xi)&= \exp(c\arctan(\frac{\xi}{k}-t)), \\
  B(t,k, \xi)&= \exp\left(-\int_0^t \frac{1}{\sqrt{1+(\frac{\xi}{k}-\tau)^2} 1_{|\frac{\xi}{k}-\tau|\leq \nu^{-1/3}}d\tau} \right),
\end{align*}
and for simplicity of notation write
\begin{align}
  \label{eq:25}
  M:=AB.
\end{align}

Let us first study the time-derivative of
\begin{align*}
  \|M \omega_{\neq}\|_{H^N}^2.
\end{align*}

Then it hold that
\begin{align}
  \label{eq:23}
  \begin{split}
  \|M \omega_{\neq}(T)\|_{H^N}^2 - \int_0^T \langle M \omega_{\neq}, \dot{M} \omega_{\neq} \rangle + \nu \int_0^T \|M (\p_y-t\p_x)\omega_{\neq}\|_{H^N}^2 \\
  = \|M \omega_{\neq}(0)\|_{H^N}^2
  + \int_0^T\langle M \omega_{\neq}, M( v_{=}^1\p_x \omega_{\neq}) \rangle\\
  + \int_0^T\langle M \omega_{\neq}, M (v_{\neq}^2 \p_y \omega_{=}) \rangle\\
  + \int_0^T \langle M \omega_{\neq}, M \p_x \theta_{\neq}\rangle\\
  + \int_0^T \langle M \omega_{\neq}, M (v_{\neq}\cdot \nabla_t \omega_{\neq}) \rangle\\
  =:\|M \omega_{\neq}(0)\|_{H^N}^2 + \mathcal{T}_{v_{=}^1} +   \mathcal{T}_{\omega_{=}} +  \mathcal{T}_{\omega_{\neq}, \theta_{\neq}} + \mathcal{T}_{v_{\neq}}.
  \end{split}
\end{align}
By assumption $\|M \omega_{\neq}(0)\|_{H^N}^2$ is much smaller than
$\epsilon^2$, so if we can show that the various terms $\mathcal{T}$ on the
right-hands-side can be controlled by the left-hand-side and higher powers of
$\epsilon$, we can show that the left-hand-side remains smaller than $4
\epsilon$ for all times.

\underline{Estimating $\mathcal{T}_{v^1_{=}}$:}
In order to estimate $\mathcal{T}_{v^1_{=}}$ we make use of cancellation in an
integration parts, following a similar argument as in \cite{zillinger2020enhanced} with additional adjustments
to account for partial dissipation.
More precisely, given the multiplier $M$, we note that by Parseval's identity
 \begin{align*}
   &\quad \langle M \omega_{\neq}, M (v_{=}^1 \p_x \omega_{\neq}) \rangle\\
   &= \langle M \omega_{\neq}, M (v_{=}^1 \p_x \omega_{\neq}) - v_{=}^1 \p_x M \omega_{\neq}\rangle\\
   &= \sum \int \int M(t,k,\xi) \omega_{\neq}(k,\xi)\omega_{\neq}(k,\xi+\zeta) (M(t,k,\xi)- M(t,k,\xi+\zeta)) v_{=}(\zeta). 
 \end{align*}
This cancellation is required to control $v_{=}(\zeta)=
\frac{1}{i\zeta}\omega_{=}(\zeta)$ in terms of $\omega_{=}$.
In particular, if $|\zeta|\geq 1$ this control is trivial, while for
$|\zeta|\leq 1$ we observe that $M(t,k,z)$ is Lipschitz with respect to $z$
uniformly in $t$ and $k\in \Z\setminus\{0\}$:
\begin{align*}
  |M(t,k,\xi)- M(t,k,\xi+\zeta)| \leq C |\zeta|.
\end{align*}
Hence, we can control $\mathcal{T}_{v^1_{=}}$ by
\begin{align*}
  \|M \omega_{\neq}\|_{L^2 H^N} \|\omega_{\neq}\|_{L^2 H^N} \|\omega_{=}\|_{L^\infty H^N}.
\end{align*}
The last factor is controlled by the preceding argument.
For the first two factors, we make the observation that
\begin{align*}
  \nu^{1/3} \leq \nu (k^2+ (\xi-kt)^2) + \frac{1}{\sqrt{1+(\frac{\xi}{k}-t)^2}} 1_{|\frac{\xi}{k}-t|\leq \nu^{-1/3}} 
\end{align*}
and hence $\|M \omega_{\neq}\|_{H^N}^2$ (and $\nu^{2/3}\|\omega_{\neq}\|_{H^N}^2$) can be estimated in terms of the dissipation and the
decay due to $\dot{M}$, at a loss of a factor $\nu^{1/3}$.\\

\underline{Estimating $\mathcal{T}_{\omega_{=}}$:}
We next discuss
\begin{align*}
 \mathcal{T}_{\omega_{=}}= \langle M \omega_{\neq}, M (v_{\neq}^2 \p_y \omega_{=}) \rangle.
\end{align*}
Here we may easily estimate by
\begin{align*}
  \nu^{-2/3} \|M \omega_{\neq}\|_{L^\infty H^N} \|M v_{\neq}^2 \|_{L^2 H^N} \|\p_y \omega_{=}\|_{L^2 H^N},
\end{align*}
where the factor of $\nu^{-2/3}$ corresponds to a rough bound of the operator
norm of $M$.
All factors are controlled in terms of the bootstrap assumption and thus
$\mathcal{T}_{\omega_{=}}$ is much smaller than $\epsilon^2$ provided
$\epsilon^3$ is much smaller than $\epsilon^2$ in terms of powers of $\nu$.

\underline{Estimating $\mathcal{T}_{\omega_{\neq}, \theta_{\neq}}$:}
As one of the main results of Section \ref{sec:nonaffine} we have shown that
$M=AB$ is constructed in just such a way that
\begin{align*}
   |\langle AB \omega_{\neq}, AB \p_x \theta_{\neq}\rangle| \leq - \langle M \omega_{\neq}, \dot{M} \omega_{\neq} \rangle -  \alpha^{-1} \langle M \hl \theta_{\neq}, \dot{M} \hl \theta_{\neq} \rangle 
\end{align*}
with $\alpha=\max(\|T'\|, \nu^{1/3})$ (see Theorem \ref{thm:bad} for the precise
definition).
Hence, we can absorb this contribution into the left-hand-side of \eqref{eq:23},
provided we can control $\dot{M} \hl \theta_{\neq}$, which will be the
left-hand-side of a later equation \eqref{eq:24}.

\underline{Estimating $\mathcal{T}_{v_{\neq}, \theta_{\neq}}$:}
It remains to discuss the main nonlinearity, where a key challenge is given by
the lack of horizontal dissipation.

If we had full dissipation at our disposal, this estimate would reduce to
controlling by
\begin{align*}
  \|\omega_{\neq}\|_{L^\infty H^N}\|v_{\neq}\|_{L^2 H^N} \|\nabla_t \omega_{\neq}\|_{L^2 H^N}.
\end{align*}
However, as we only require vertical dissipation the last factor is not easily
controlled anymore. We thus have to invest additional effort to control this
contribution.

As $v_{\neq}$ is divergence-free, we observe that
\begin{align*}
  \langle M \omega_{\neq}, M(v_{\neq} \cdot \nabla_t \omega_{\neq}) \rangle
  &= \langle M \omega_{\neq}, M(v_{\neq} \cdot \nabla_t \omega_{\neq}) - v_{\neq} \cdot \nabla_t M \omega_{\neq} \rangle\\
  &= \sum \iiint M(k,\xi) \tilde{\omega}_{\neq}(k,\xi) (M(k,\xi)- M(k-l,\xi-\zeta)) \tilde{v}_{\neq}(l,\zeta) \\
  & \quad \cdot
  \begin{pmatrix}
    k-l \\
    \xi+\zeta - (k-l)t
  \end{pmatrix}
  \tilde{\omega}_{\neq}(k-l,\xi-\zeta).
\end{align*}
We observe that if
\begin{align*}
  |k-l|\leq \nu^{-1} |\xi+\zeta - (k-l)t|
\end{align*}
the last gradient can simply be controlled by the vertical dissipation, which
yields an estimate in terms of
\begin{align*}
  \|\omega_{\neq}\|_{L^\infty H^N}\|v_{\neq}\|_{L^2 H^N} \|(\p_y-t\p_x) \omega_{\neq}\|_{L^2 H^N}
\end{align*}
and can hence be controlled.
Similarly, if
\begin{align*}
  |k-l|\leq \nu^{-1} |l|
\end{align*}
we can control in terms of
\begin{align*}
   \|\omega_{\neq}\|_{L^\infty H^N}\|\p_x v_{\neq}\|_{L^2 H^N} \|\omega_{\neq}\|_{L^2 H^N}.
\end{align*}
It thus only remains to discuss the region where
\begin{align}
  \label{eq:20}
  \begin{split}
  |t- \frac{\xi+\zeta}{k-l}|&\leq \nu, \\
  |l|&\leq \nu |k|.
  \end{split}
\end{align}
Here, we make use of cancellations in $M$.
More precisely, we note that $M(k,\xi)$ does not depend on $k$ and $\xi$
individually, but only on $\frac{\xi}{k}$ and that uniformly in time
\begin{align*}
  |M(k,\xi)- M(k-l,\xi-\zeta)| \leq C |\frac{\xi}{k}- \frac{\xi-\zeta}{k-l}|\\
  = C |\frac{\xi-kt}{k}- \frac{\xi-\zeta- (k-l)t}{k-l}|\\
  \leq C \frac{1}{1+\nu} \frac{1}{|k-l|}(|\xi-kt| + |\xi-\zeta-(k-l)t|),
\end{align*}
where we used \eqref{eq:20}.
We thus can control $\mathcal{T}_{v_{\neq}, \theta_{\neq}}$ in that region by
\begin{align*}
  \|\omega_{\neq}\|_{L^\infty H^N}\|v_{\neq}\|_{L^2 H^N}\|(\p_y-t\p_x) \omega_{\neq}\|_{L^2 H^N},
\end{align*}
which concludes the argument.

\underline{Controlling $\hl \theta_{\neq}$}
We next turn to controlling $\hl \theta_{\neq}$, where we study the time
derivative of
\begin{align*}
  \|M \hl \theta_{\neq}\|_{H^N}^2.
\end{align*}
Integrating in time, we have to control
\begin{align}
  \label{eq:24}
  \begin{split}
 \|M \hl \theta_{\neq}(T)\|_{H^N}^2 - \int_0^T \langle M \hl \theta_{\neq}, \dot{M} \hl \theta_{\neq} \rangle + \nu \int_0^T \|M (\p_y-t\p_x)\hl \theta_{\neq}\|_{H^N}^2 \\
  = \|M \hl \theta_{\neq}(0)\|_{H^N}^2 + \int_0^T \langle M \hl \theta_{\neq}, M \hl T'(y) v_{\neq}^2\rangle
  + \int_0^T \langle M \hl \theta_{\neq}, M \hl v_{=}^1 \p_x \theta_{\neq}  \rangle \\
  + \int_0^T \langle M \hl \theta_{\neq}, M \hl v_{\neq}^2 \p_y \theta_{=} \rangle \\
  + \int_0^T \langle M \hl \theta_{\neq},  M \hl v_{\neq}\cdot \nabla_t \theta_{\neq}\rangle\\
  =: \|M \hl \theta_{\neq}(0)\|_{H^N}^2 + \mathcal{T}_{T}+ \mathcal{T}_{v_{=}^1} + \mathcal{T}_{\theta_{=}} + \mathcal{T}_{v_{\neq}, \hl \theta_{\neq}}.
  \end{split}
\end{align}
Here the aim again is to to show that that all $\mathcal{T}$ contributions add
up to something smaller than $\epsilon^2$ and hence equality is not attained.

\underline{Estimating $\mathcal{T}_{T}$}
As one of the main results of Section \ref{sec:nonaffine} we have shown that
$\mathcal{T}_{T}$ can be controlled in terms of the decay of the multipliers $M$
and the vertical dissipation of $\omega$ only. Thus this contribution can
estimated in terms of the left-hand-side of \eqref{eq:24} and \eqref{eq:23}.

\underline{Estimating $\mathcal{T}_{v_{=}^1}$:}
Here we may argue analogously as for $\omega_{\neq}$, expect that $M$ has been
replaced by $\hl M$. We thus obtain an estimate by
\begin{align*}
  \|M \hl \theta_{\neq}\|_{L^2 H^N} \|\theta_{\neq}\|_{L^2 H^N} \|\omega_{=}\|_{L^\infty H^N}.
\end{align*}

\underline{Estimating $\mathcal{T}_{\theta_{=}}$:}
Here we may argue again analogously as for $\omega_{\neq}$ and control by
\begin{align*}
  \|M \hl \theta_{\neq}\|_{L^\infty H^N}\|v_{\neq}^2\|_{L^2 H^N} \|\p_y \theta_{=}\|_{L^2 H^N}.
\end{align*}

\underline{Estimating $\mathcal{T}_{v_{\neq}, \hl \theta_{\neq}}$:}
We recall that in this theorem we assume vertical dissipation also for the temperature (in contrast to
Section \ref{sec:nonaffine} and the problem considered in
\cite{masmoudi2020stability}).
Therefore, in this estimate we argue largely analogously to to the estimate of
$\mathcal{T}_{v_{\neq}, \omega_{\neq}}$.
However, since $\hl M$ also depends on $k$, we need some additional control in
the region where the horizontal dissipation is not easily controlled.

More precisely, by the preceding arguments for $\mathcal{T}_{v_{\neq},
  \omega_{\neq}}$ it suffices to consider
\begin{align*}
  &\sum \iint (\hl \theta_{\neq})(k,\xi) \frac{1}{\hl(k-l,\xi-\zeta)} (\hl M (k,\xi)- \hl M (k-l,\xi-\zeta))\tilde{v}_{\neq}(l,\zeta) \\
  & \quad \cdot
  \begin{pmatrix}
    k-l\\
    \xi-\zeta +(k-l)t
  \end{pmatrix}
  (\hl \theta_{\neq})(k-l,\xi-\zeta)
\end{align*}
in the regions where $\xi-\zeta$ is very close to resonant and $l$ is much
smaller than $k$.

However, in that case we may split into differences in $M$ and in $\hl$ and
observe that
\begin{align*}
  \frac{\sqrt{k^2+ (\xi-kt)^2} - \sqrt{(k-l)^2+ (\xi+\zeta- (k-l)t)^2}}{\sqrt{(k-l)^2+ (\xi+\zeta- (k-l)t)^2}}\\
  \approx \frac{\sqrt{k^2}- \sqrt{(k-l)^2}}{\sqrt{(k-l)^2}} \approx \frac{l}{|k-l|},
\end{align*}
where we could neglect $\xi-kt$ and $\xi+\zeta- (k-l)t$ since these terms could
otherwise be controlled in terms of the  vertical dissipation.
Hence, over all we can control by
\begin{align*}
  \|\hl \theta_{\neq}\|_{L^\infty H^N}\|\p_x v_{\neq}\|_{L^2 H^N} \|\hl \theta_{\neq}\|_{L^2 H^N},
\end{align*}
which concludes the proof.

\end{proof}

\subsubsection*{Acknowledgments}
This research has been funded by the Deutsche Forschungsgemeinschaft (DFG, German Research Foundation) – Project-ID 258734477 – SFB 1173.

The author would like to thank Charlie Doering for pointing out the questions of
Rayleigh-B\'enard instability without thermal diffusion for
profiles other than hydrostatic balance.

\bibliographystyle{alpha}
\bibliography{citations2}

\end{document}